\documentclass[a4paper, 11pt]{amsart}   
\usepackage{mathptmx, amssymb,amscd,latexsym, eulervm}   
\usepackage{amsmath, mathtools} 
\usepackage{amsthm}
\usepackage{mathdots}
\usepackage[pagebackref,colorlinks=true,linkcolor=blue,urlcolor=blue]{hyperref}
\usepackage{color}
\usepackage[onehalfspacing]{setspace}
\usepackage{tabularx}
\usepackage{amsfonts}
\usepackage{paralist}
\usepackage{aliascnt}
\usepackage[initials, lite]{amsrefs}
\usepackage{amscd}
\usepackage{blkarray}
\usepackage{mathbbol}
\usepackage{setspace}
\usepackage[inner=2.4cm,outer=2.4cm, bottom=3.2cm]{geometry}
\usepackage{tikz, tikz-cd}
\usepackage{calligra,mathrsfs}

\usepackage{tikz}
\usetikzlibrary{matrix}
\usetikzlibrary{arrows,calc}
\allowdisplaybreaks

\BibSpec{collection.article}{%
	+{}  {\PrintAuthors}                {author}
	+{,} { \textit}                     {title}
	+{.} { }                            {part}
	+{:} { \textit}                     {subtitle}
	+{,} { \PrintContributions}         {contribution}
	+{,} { \PrintConference}            {conference}
	+{}  {\PrintBook}                   {book}
	+{,} { }                            {booktitle}
	+{,} { }                            {series}
	+{, vol.} { }                            {volume}
	+{,} { }                            {publisher}
	+{,} { \PrintDateB}                 {date}
	+{,} { pp.~}                        {pages}
	+{,} { }                            {status}
	+{,} { \PrintDOI}                   {doi}
	+{,} { available at \eprint}        {eprint}
	+{}  { \parenthesize}               {language}
	+{}  { \PrintTranslation}           {translation}
	+{;} { \PrintReprint}               {reprint}
	+{.} { }                            {note}
	+{.} {}                             {transition}
	+{}  {\SentenceSpace \PrintReviews} {review}
}
\AtBeginDocument{%
	\def\MR#1{}
}

\makeatletter
\@namedef{subjclassname@2020}{%
	\textup{2020} Mathematics Subject Classification}
\makeatother

\newcommand{\kk}{\mathbb{k}}
\newcommand{\KK}{\mathbb{K}}
\newcommand{\bFF}{\mathbb{F}}

\newcommand{\CC}{\mathbb{C}}

\newcommand{\NN}{\normalfont\mathbb{N}}
\newcommand{\ZZ}{\mathbb{Z}}

\newcommand{\MM}{{\normalfont\mathfrak{M}}}

\newcommand{\xx}{{\normalfont\mathbf{x}}}

\newcommand{\dd}{\mathbb{d}}

\newcommand{\mm}{{\normalfont\mathfrak{m}}}
\newcommand{\init}{{\normalfont\text{in}}}

\newcommand{\QQ}{\mathbb{Q}}
\newcommand{\pp}{{\normalfont\mathfrak{p}}}
\newcommand{\aaa}{\mathfrak{a}}
\newcommand{\bbb}{\mathfrak{b}}

\newcommand{\bn}{{\normalfont\mathbf{n}}}

\newcommand{\bm}{{\normalfont\mathbf{m}}}

\newcommand{\nnn}{{\normalfont\mathfrak{n}}}

\newcommand{\fJ}{{\mathfrak{J}}}

\newcommand{\Ker}{\normalfont\text{Ker}}

\newcommand{\Quot}{\normalfont\text{Quot}}

\newcommand{\Syz}{\normalfont\text{Syz}}

\newcommand{\Ass}{{\normalfont\text{Ass}}}

\newcommand{\Sym}{\normalfont\text{Sym}}

\newcommand{\Hom}{\normalfont\text{Hom}}

\newcommand{\LL}{\mathbb{L}}

\newcommand{\HL}{\normalfont\text{H}_{\mm}}
\newcommand{\HH}{\normalfont\text{H}}

\newcommand{\iniTerm}{\normalfont\text{in}}

\newcommand{\bideg}{\normalfont\text{bideg}}

\newcommand{\Spec}{\normalfont\text{Spec}}
\newcommand{\Diff}{{\normalfont\text{Diff}}}

\newcommand{\Princ}{P_{R/A}}
\newcommand{\MaxSpec}{\normalfont\text{MaxSpec}}

\DeclareMathOperator{\lc}{H}

\DeclareMathOperator{\Der}{Der}
\DeclareMathOperator{\gr}{gr}
\DeclareMathOperator{\IM}{Im}
\newcommand{\fps}{{\mathbb {f}^s}}

\def\f0{\mathbf{0}}

\def\1{\mathbf{1}}

\newtheorem*{blanktheorem}{Theorem}
\newtheorem*{GFL}{Generic Freeness Lemma}
\newtheorem{theorem}{Theorem}[section]

\newtheorem{headthm}{Theorem}

\newaliascnt{headcor}{headthm}

\aliascntresetthe{headcor}

\newaliascnt{headconj}{headthm}

\aliascntresetthe{headconj}

\newaliascnt{corollary}{theorem}
\newtheorem{corollary}[corollary]{Corollary}
\aliascntresetthe{corollary}

\newaliascnt{claim}{theorem}

\aliascntresetthe{claim}

\newaliascnt{lemma}{theorem}
\newtheorem{lemma}[lemma]{Lemma}
\aliascntresetthe{lemma}

\newaliascnt{conjecture}{theorem}

\aliascntresetthe{conjecture}

\newaliascnt{proposition}{theorem}
\newtheorem{proposition}[proposition]{Proposition}
\aliascntresetthe{proposition}

\theoremstyle{definition}
\newaliascnt{definition}{theorem}
\newtheorem{definition}[definition]{Definition}
\aliascntresetthe{definition}

\newaliascnt{notation}{theorem}

\aliascntresetthe{notation}

\newaliascnt{example}{theorem}

\aliascntresetthe{example}

\newaliascnt{examples}{theorem}

\aliascntresetthe{examples}

\newaliascnt{remark}{theorem}
\newtheorem{remark}[remark]{Remark}
\aliascntresetthe{remark}

\newaliascnt{question}{theorem}
\newtheorem{question}[question]{Question}
\aliascntresetthe{question}

\newaliascnt{questions}{theorem}

\aliascntresetthe{questions}

\newaliascnt{problem}{theorem}

\aliascntresetthe{problem}

\newaliascnt{construction}{theorem}

\aliascntresetthe{construction}

\newaliascnt{setup}{theorem}
\newtheorem{setup}[setup]{Setup}
\aliascntresetthe{setup}

\newaliascnt{setupdef}{theorem}
\newtheorem{setupdef}[setupdef]{Setup-Definition}
\aliascntresetthe{setupdef}

\newaliascnt{algorithm}{theorem}

\aliascntresetthe{algorithm}

\newaliascnt{observation}{theorem}
\newtheorem{observation}[observation]{Observation}
\aliascntresetthe{observation}

\newaliascnt{defprop}{theorem}

\aliascntresetthe{defprop}

\def\equationautorefname~#1\null{(#1)\null}
\def\sectionautorefname~#1\null{Section #1\null}
\def\subsectionautorefname~#1\null{\S #1\null}

\def\surjects{\twoheadrightarrow}
\title{Effective generic freeness and applications to local cohomology}

\author{Yairon Cid-Ruiz}
\address[Cid-Ruiz]{Department of Mathematics, North Carolina State University, Raleigh, NC 27695, USA}
\email{ycidrui@ncsu.edu}

\author{Ilya Smirnov}
\address[Smirnov]{BCAM -- Basque Center for Applied Mathematics, Bilbao, Spain \quad and \quad IKERBASQUE, Basque Foundation for Science, Bilbao, Spain}
\email{ismirnov@bcamath.org}

\date{\today}

\begin{document}

	\maketitle

	\begin{abstract}
		Let $A$ be a Noetherian domain and $R$ be a finitely generated $A$-algebra.
		We study several features regarding the generic freeness over $A$ of an $R$-module.
		For an ideal $I \subset R$, we show that the local cohomology modules $\HH_I^i(R)$ are generically free over $A$ under certain settings where $R$ is a smooth $A$-algebra.
		By utilizing the theory of Gr\"obner bases over arbitrary Noetherian rings, we provide an effective method to make explicit the generic freeness over $A$ of a finitely generated $R$-module. 
	\end{abstract}

	\section{Introduction}
	
	Our starting point is the following classical result of Grothendieck: 

\begin{GFL}
Let $A$ be a Noetherian domain, $R$ be a finitely generated $A$-algebra, and $M$ be a finitely generated $R$-module. 
There is a nonzero element $a \in A$ such that $M \otimes_A A_a$ is a free $A_a$-module.
\end{GFL}

%	Let $A$ be a Noetherian domain, $R$ be a finitely generated $A$-algebra, and $M$ be a finitely generated $R$-module. 
%	Our starting point is the following classical result: 	
	
%	\begin{changemargin}{.7cm}{.7cm} 
%	\textit{Grothendieck's Generic Freeness Lemma}: There is a nonzero element $a \in A$ such that $M \otimes_A A_a$ is a free $A_a$-module.
%	\end{changemargin} 
%
\noindent
	The main goal of this paper is to extend the above result in the following two directions: 
	\begin{enumerate}[\rm (a)]
		\item\label{goal_a} We extend the above result when $M = \HH_I^i(R)$ is a local cohomology module with support on an ideal $I \subset R$, $A$ contains a field, and $R$ is a smooth $A$-algebra (recall that local cohomology modules are typically not finitely generated).
		\smallskip
		\item\label{goal_b} 
		We give an effective and computable method to choose a specific element $a \in A$ in the above result.
	\end{enumerate}

	\noindent
	The generic freeness of local cohomology modules is an important problem that has been addressed in various contexts by Hochster and Roberts  \cite[Theorem 3.4]{HOCHSTER_ROBERTS_PURITY}, by Koll\'ar \cite[Theorem 78]{kollar2014maps}, and by Smith \cite{KSMITH}. 
	Also, see the recent papers \cite{GEN_FREENESS_LOC_COHOM, FIBER_FULL}.
	
	We divide the Introduction into two subsections addressing these two goals.

	\subsection*{Effective generic freeness}
	Our primary tool to address  \hyperref[goal_b]{Goal~(b)} is the theory of Gr\"obner bases, where we use $A$ as the coefficient ring. 
	The study of Gr\"obner bases over an arbitrary Noetherian coefficient ring (and not just a field) is a well-established and useful technique (see \cite[Chapter 4]{ADAMS_LOUST_GROEBNER}, \cite{VASCONCELOS_FLATNESS_EFF}, \cite{kalkbrener1998algorithmic}, \cite{miller1996analogs}, \cite{HAIMAN_STURMFELS}).
	
	Suppose that $R = A[x_1,\ldots,x_r]$ is a polynomial ring over $A$ and $>$ is a monomial order on $R$ (see \autoref{sect_Grobner} for more details).
	Vasconcelos~\cite{VASCONCELOS_FLATNESS_EFF} noticed that an effective method of determining generic freeness is to compute an initial ideal and then to make invertible the leading coefficients obtained in $A$. A similar idea was later used by Derksen and Kemper in \cite{DerkKemp}.
	
%	\begin{changemargin}{.7cm}{.7cm} 	
	\begin{blanktheorem}[Vasconcelos \cite{VASCONCELOS_FLATNESS_EFF}]
		Let $I \subset R$ be an ideal and $\init_>(I) = \big(a_1\xx^{\beta_1},\, a_2\xx^{\beta_2}, \,\ldots\,, a_m\xx^{\beta_m} \big)$ be its corresponding initial ideal, where $0\neq a_i \in A$ and $\beta_i \in \NN^r$.
		Choose $0 \neq a \in (a_1) \cap \cdots \cap (a_b) \subset A$.
		Then $R/I \otimes_A A_a$ is a free $A_a$-module. 
	\end{blanktheorem}
%		\textit{Vasconcelos \cite{VASCONCELOS_FLATNESS_EFF}}:

%	\end{changemargin}

	In \autoref{thm_gen_grob}, we extend the above result to the case of modules, and we provide \emph{generic versions}  (extended to our setting over the coefficient ring $A$) of two important results: \emph{Macaulay's theorem} and \emph{flat degenerations to the initial ideal/module}.
	As a consequence of this result, we obtain new effective proofs for Grothendieck's generic freeness lemma (see \autoref{cor: gen free}) and for a stronger version by Hochster and Roberts (see \autoref{prop_hochster_roberts}).

	Suppose now that $R = A[x_1,\ldots,x_r]$ is a positively graded polynomial ring over $A$ and $\mm = (x_1,\ldots,x_r) \subset R$ is the graded irrelevant ideal.
	We also provide an effective method of determining generic freeness for the local cohomology modules $\HL^i(R/I)$, under the assumption that the corresponding initial ideal is ``square-free'' over the coefficient ring $A$. 
	This generic freeness result is inspired by the work of Conca and Varbaro \cite{CONCA_VARBARO} on square-free Gr\"obner degenerations.

	\begin{headthm}[\autoref{thm_sqr_free_deg}]
		\label{thmA}
		Assume that $A$ contains a field $\kk$.
		Let $I \subset R$ be a homogeneous ideal and
		$\iniTerm_>(I) = \big(a_1\xx^{\beta_1}, \, a_2\xx^{\beta_2},\, \ldots\,, a_b\xx^{\beta_b} \big)$ be its corresponding initial ideal, where $0\neq a_i \in A$ and $\beta_i \in \NN^r$.
		Suppose that each monomial $\xx^{\beta_i}$ is square-free {\rm(}i.e., $\beta_i = (\beta_{i,1}, \ldots,\beta_{i,r}) \in \NN^r$ with $\beta_{i,j} \le 1${\rm)}.
		Choose $0 \neq a \in (a_1) \cap \cdots \cap (a_b) \subset A$.
		Then $\HL^i(R/I) \otimes_{A} A_a$ is $A_a$-free for all $i \ge 0$.
	\end{headthm}

	The assumption that $A$ contains a field cannot be dropped in \autoref{thmA} (see \autoref{rem_triang_PP_2}).
	In \autoref{sect_determinant}, we apply these results to study certain specializations of determinantal ideals.

	\subsection*{Applications to local cohomology}
	We now describe our contributions towards \hyperref[goal_a]{Goal (a)}.
	To the best of our knowledge, the previous most comprehensive result regarding the generic freeness of local cohomology modules is the following:

	\begin{blanktheorem}[Smith \cite{KSMITH}]
		Let $I \subset R$ be an ideal such that \emph{$R/I$ is a finitely generated $A$-module}. 
		Then there is a nonzero element $a \in A$ such that $\HH_I^i(M) \otimes_A A_a$ is a free $A_a$-module for all $i \ge 0$.
	\end{blanktheorem}
%	\begin{changemargin}{.7cm}{.7cm} 
%		\textit{}: 
%	\end{changemargin} 
	
	\noindent
	A typical example where the above result applies is when $R = A[x_1,\ldots,x_r]$ is a polynomial ring and $I = (x_1,\ldots,x_r) \subset R$ (e.g., as in \autoref{thmA}).
	If we drop the assumption that $\Spec(R/I) \rightarrow \Spec(A)$ is a finite morphism, then there are known examples where one does \emph{not} have generic freeness for local cohomology modules (see \autoref{rem_Katzman}).
	
	In the following theorem, for any ideal $I \subset R$, we settle the generic freeness of $\HH_I^i(R)$ under certain assumptions that include the smoothness of the morphism $\Spec(R) \rightarrow \Spec(A)$.
	
	\begin{headthm}[\autoref{thm_local_cohom}]
		\label{thmB}
		Assume that $A$ contains a field $\kk$ and $R$ is a smooth $A$-algebra. 
		Suppose one of the following two conditions: 
		\begin{enumerate}[\rm (a)]
			\item $\kk$ is a field of characteristic zero, or
			\item $\kk$ is a field of positive characteristic and the regular locus ${\rm Reg}(A) \subset \Spec(A)$ contains a nonempty open subset.
		\end{enumerate}
		Then, for any ideal $I \subset R$, there is a nonzero element $a \in A$ such that $\HH_I^i(R) \otimes_A A_a$ is a free $A_a$-module for all $i \ge 0$. 
	\end{headthm}

	Notice that \autoref{thmB} is applicable in the particular case when $A$ is an algebra finitely generated over any field $\kk$ (see \autoref{rem_reg_locus}). 
	In light of \autoref{thmB}, it is natural to ask about generic freeness of local cohomology in more general settings.
	Due to \cite[Example 3.3]{BBLSZ}, there are smooth algebras over $\ZZ$ where the local cohomology $\HH_I^i(R)$ is generically flat but not generically free over $\ZZ$.
	Therefore we should ask the following question:
	
	\begin{question}
		Suppose that $R$ is a smooth $A$-algebra.
		Let $I \subset R$ be an ideal and $i \ge 0$.
		Does there exist a nonzero element $0 \neq a \in A$ such that $\HH_I^i(R) \otimes_{A} A_a$ is a flat $A_a$-module?
	\end{question}

	\noindent
	We have a partial answer when $A = \ZZ$: due to \cite{BBLSZ} we know that $\HH_I^i(R)$ has finitely many associated primes, and hence by inverting a suitable element $0 \neq a \in \ZZ$, we can make $\HH_I^i(R)$ torsion-free and thus $\ZZ$-flat.
	Furthermore, the argument of \cite{BBLSZ} allows to deduce finiteness of associated primes from generic flatness. 
	Thus, we settle a new case of Lyubeznik's conjecture: \autoref{thm_finite_ass} gives finiteness of associated primes of local cohomology in a smooth algebra over a Dedekind domain of characteristic $0$.	
	
	Our proof of \autoref{thmB} is inspired by the works of Lyubeznik \cite{Lyubeznik_DMOD, Lyubeznik_FMOD} on the finiteness of associated primes of local cohomology. 
	For the positive characteristic case of \autoref{thmB}, the theory of $F$-modules (as developed by Lyubeznik \cite{Lyubeznik_FMOD}) is our main tool (see \autoref{subsect_pos_char}).
	
	For the characteristic zero case of \autoref{thmB}, we obtain several results regarding the ring of differential operators $D_{R/A}$.
	For instance, under the assumption that $R$ is a smooth $A$-algebra, we show that $D_{R/A}$ equals the derivation ring $\Delta(R/A)$, and that for any $f \in R$ the localization $R_f$ is \emph{generically} a finitely generated left module over $D_{R/A}$.
	These type of results are classical when $A$ is a field (see \cite[Chapter 15]{McCONNELL_ROBSON}).
	Our main result in this direction is the following theorem.

	\begin{headthm}[\autoref{thm_gen_prop_smooth_A}]
		\label{thmC}
		Assume that $A$ contains the field $\QQ$ of rational numbers and $R$ is a smooth $A$-algebra.
		Then the following statements hold: 
		\begin{enumerate}[\rm (i)]
			\item $D_{R/A} = \Delta(R/A)$. 
			In particular,
			$
			\gr(D_{R/A}) = \bigoplus_{m=0}^\infty D^{m}_{R/A}/D^{m-1}_{R/A}
			$ 
			is a Noetherian commutative ring and $D_{R/A}$ is a Noetherian ring.
			\item $D_{R/A}$ is strongly right Noetherian.
			\item 
			For any $f \in R$, there exists a nonzero element $a \in A$ such that $R_f \otimes_{A} A_a$ is a finitely generated left module over $D_{R/A} \otimes_{A} A_a$.
		\end{enumerate}
	\end{headthm}
			
	Our proof of \autoref{thmC} is based on the following ideas. 
	From the smooth morphism $\Spec(R) \rightarrow \Spec(A)$, we obtain an affine open covering $\Spec(R_{g_i})$ of $\Spec(R)$ with \'etale morphisms $\Spec(R_{g_i}) \rightarrow \mathbb{A}_A^{n_i}$ (see \autoref{rem_cover_smooth_etale}).
	The ring differential operators $D_{P/A}$ of a polynomial ring $P$ over $A$ is well-understood (it is a relative Weyl algebra, see \autoref{rem_relative_Weyl_alg}), and then we can determine the ring of differential operators $D_{R_{g_i}/A}$ by studying the behavior of differential operators under an \'etale ring map (see \autoref{thm_diff_ops_etale}).
	This study of differential operators under \'etale ring maps follows from the work of M\'asson \cite{MASSON}.
	The finiteness result in part (iii) of \autoref{thmC} is a consequence of the rationality of the roots of Bernstein--Sato polynomials (see \autoref{prop_finite_gen_loc}, \autoref{thm: BS rationality}).
	
	\subsection*{Outline}
	The basic outline of this paper is as follows.
	In \autoref{sect_Grobner}, we give our approach to Gr\"obner bases over a Noetherian commutative coefficient ring.
	We provide the proof of \autoref{thmA} in \autoref{sect_sqr_grob}.
	In \autoref{sect_gen_free_filtered}, we study the generic freeness of certain not necessarily commutative filtered algebras.
	In \autoref{sect_diff_ops_smooth}, we study differential operators in a characteristic zero smooth setting and we prove \autoref{thmC}.
	The proof of \autoref{thmB} is given in \autoref{sect_local_cohom}.
	Finally, \autoref{sect_determinant} shows how to apply some of the methods developed in this paper to study certain specializations of determinantal ideals.
	
	\smallskip
	
	\noindent
	\textbf{Convention.} Unless specified otherwise, a ring is always assumed to be commutative. 
	All rings are assumed to be unitary.
	
	\section{Gr\"obner bases  \& Grothendieck's generic freeness lemma} 
	\label{sect_Grobner}

	In this section, we quickly review the basic theory of Gr\"obner bases over Noetherian rings with the aim of obtaining an effective proof of Grothendieck's generic freeness lemma.
	For more details regarding the theory of Gr\"obner bases over arbitrary Noetherian rings, the reader is referred to \cite[Chapter 4]{ADAMS_LOUST_GROEBNER}.
	The following setup and definitions are used throughout this section.

	\begin{setupdef}
		\label{setup_Grobner_deform}
		Let $A$ be a Noetherian domain and  $R = A[x_1,\ldots,x_r]$ be a polynomial ring over $A$.
		Let $F$ be a  free $R$-module $F = \bigoplus_{i=1}^\ell R e_i$ and $>$ be a \emph{monomial order on $F$}.
		A \emph{monomial in $F$} is an element of the form $m = \xx^\bn e_i = x_1^{n_1}\cdots x_r^{n_r} e_i \in F$ and a \emph{term in $F$} is an element of the form $m'  = am \in F$,  where $a \in A$, $\bn = (n_1,\ldots,n_r) \in \NN^r$ and $1 \le i \le \ell$.
		A \emph{monomial submodule of $F$} is a submodule generated by monomials.
		The monomial order $>$ is characterized by the following two conditions: 
		\begin{enumerate}[\rm (i)]
			\item $>$ is a total order on the set of all monomials in $F$.
			\item If $m_1$, $m_2$ are monomials in $F$ and $n\ne1$ is a monomial in $R$, then 
			$$
			m_1 > m_2 \qquad  \text{ implies that } \qquad nm_1 > nm_2 > m_2.
			$$ 
		\end{enumerate}
		We will always assume that the monomial order on $F$ is \emph{compatible} with a monomial order on $R$ (that we also denote as $>$) in the following sense: $\xx^\bn e_i > \xx^\bm e_j$ if and only if $i < j$, or $i = j$ and $\xx^\bn > \xx^\bm$.
		An element $w\in F$ can be written \emph{uniquely} as $w = \sum_{i = 1}^b a_im_i$ where the $m_i$'s are different monomials, and  its \emph{initial term} $\init_>(w)$ is given by the term with largest corresponding monomial.
		For a given $R$-submodule $M \subset F$, the corresponding \emph{initial module} is given by 
		$$
		\init_>(M) \,\coloneqq\, \left(\init_>(w) \mid w \in M\right).
		$$
		Notice that, contrary to the case where $A$ is a field, the initial submodule $\init_>(M)$ is \emph{not} necessarily a monomial submodule as it is generated by terms.
		A set of elements $\{w_1, \ldots, w_b\}$ in an $R$-submodule $M \subset F$ is said to be a \emph{Gr\"obner basis for $M$} if $\init_>(M) = \left(\init_>(w_1), \ldots, \init_>(w_b)\right)$.
		Whenever the monomial order used is clear from the context we will drop reference to it. 
		By abusing notation, we compare terms according to their respective monomials.
	\end{setupdef}

\begin{remark}
	\label{rem_basics_mons}
	The following statements hold.
	\begin{enumerate}[\rm (i)]
		\item Let $m \in F$ and $m_1, \ldots, m_b \in F$ be terms.
		If 
		$m \in (m_1, \ldots, m_b)$, then we can write $m = \sum_{i = 1}^b h_i m_i$ where the $h_i$'s are terms in $R$.
		\item If $w \in F$ and $f \in R$, then $\iniTerm(fw) = \iniTerm(f)\iniTerm(w)$.
	\end{enumerate}	
\end{remark}	
\begin{proof}
	
	(i) By the assumption we can write $m = \sum_{i = 1}^b f_i m_i$, where $f_i$ is a polynomial in $R$. 
However, after rewriting $m$ as a sum of terms and grouping 
the terms whose monomial is equal to the one of $m$, we will obtain an expression that will have (at most) one term $h_i \in R$ from each $f_i$. 

	(ii) Write decompositions $w = am + w'$ and $f = bh + f'$ where $a,b \in A$ are nonzero elements, and $m \in F$ and $h \in R$ are the maximal monomials of $w$ and $f$, respectively. 
	Since $A$ is a domain, $ab \neq 0$ and so we get $\iniTerm(fw) = am \cdot bh =  \iniTerm(f)\iniTerm(w)$.
\end{proof}	
	
	\begin{lemma}
		Let $M \subset F$ be an $R$-submodule.
		If $\{w_1,\ldots,w_b\} \in M$ is a Gr\"obner basis for $M$, then $M$ is generated by the elements $w_1,\ldots,w_b$.
	\end{lemma}
	\begin{proof}
	Let $w \in M$ be a nonzero element. 
	The assumption and \autoref{rem_basics_mons}(i) yield the existence of terms $h_1, \ldots, h_b \in R$ such that $\init(w) = \sum_{i = 1}^b h_i \init(w_i)$. 
	We have that $h_i \iniTerm(w_i) = \iniTerm(h_iw_i)$ (see \autoref{rem_basics_mons}(ii)).
	Therefore, it follows that $w' = w - \sum_{i = 1}^b h_iw_i$ satisfies $\iniTerm(w') < \iniTerm (w)$. 
	Since a monomial order satisfies the descending chain condition (see, e.g., \cite[Lemma 15.2]{EISEN_COMM}, \cite[\S 2.1]{HERZOG_HIBI_MONOMIALS}), we can show that $w \in \left(w_1,\ldots,w_b\right)$ by repeating the above procedure finitely many times. 	
	\end{proof}
	
	\begin{proposition}
		Let $M \subset F$ be an $R$-submodule.
		If an element $f \in R$ is such that $\iniTerm (f)$ is a non-zero-divisor over $F/\iniTerm(M)$, then $f$ is a non-zero-divisor over $F/M$.
	\end{proposition}
	\begin{proof}
	Suppose there is an element $w \in F$ such that $fw \in M$. 
	By \autoref{rem_basics_mons}(ii), $\iniTerm(fw) = \iniTerm(f) \iniTerm(w)$ and so the hypothesis implies that $\iniTerm(w) \in \iniTerm(M)$.
	Let $\{w_1,\ldots,w_b\}$  be a Gr\"obner basis for $M$. 
	We can find terms $h_i \in R$ such that $\iniTerm(w) = \sum_{i = 1}^b h_i\iniTerm(w_i)$. 
	Let $w' = w - \sum_{i = 1}^b h_i w_i$. 
	It is then clear that $fw' \in M$ and that $\iniTerm (w) > \iniTerm (w')$. 
	Again, since a monomial order satisfies a descending chain condition, by repeating the above process we obtain that $w \in M$.
\end{proof}

	We now discuss a process of homogenization for an $R$-submodule  $M \subset F$. 
	Let $\omega = (\omega_1, \ldots, \omega_r)\in \ZZ_+^{r}$ and $\mathbb{d} = (d_1, \ldots, d_\ell)\in \ZZ_+^{\ell}$ be two weight vectors. 
	The corresponding $(\omega,\mathbb{d})$-degree of the monomial $\xx^\bn e_k = x_1^{n_{1}} \cdots x_r^{n_{r}} e_k \in F$ is given by 
	$$
	\deg_{\omega,\mathbb{d}}(\xx^\bn e_k) \,\coloneqq\,    \omega \cdot \bn + d_k \;=\; n_{1} \omega_{1} + \cdots + n_r \omega_{r} + d_k.  
	$$
	For an element $w \in F$, $\deg_{\omega,\dd}(w)$ is the maximum $(\omega,\dd)$-degree of the terms of $w$ and  $\iniTerm_{\omega,\dd}(w)$ is the sum of all the terms of $w$ of maximal $(\omega,\dd)$-degree.
	Consider the polynomial ring $S = R[t]$ and the corresponding free $S$-module $F[t] = F \otimes_R S$.
	For an element $w =  \sum_{j} a_j \, \xx^{\bn_j} \, e_{k_j} \in F$, the corresponding $(\omega,\dd)$-homogenization is given by
	$$
	\hom_{\omega,\dd}(w) \,\coloneqq\, \sum_{j} a_j \, \xx^{\bn_j} \, t^{{\deg_{\omega,\dd}(w)} - {\deg_{\omega,\dd}(\xx^{\bn_j}e_{k_j})}} \,  e_{k_j} \; \in  F[t].
	$$
	For an $R$-submodule $M \subset F$, we define $\hom_{\omega,\dd}(M) \subset F[t]$ as  the $S$-submodule generated by $\hom_{\omega,\dd}(w)$ for all $w \in M$.
	After considering $S$ as a graded polynomial ring with $[S]_0 = A$, $\deg(x_i) = \omega_i$ and $\deg(t) = 1$, we obtain that $\hom_{\omega,\dd}(M)$ is a graded $S$-submodule of the graded free $S$-module $F[t] = \bigoplus_{i = 1}^\ell Se_i \cong \bigoplus_{i = 1}^\ell S(-d_i)$.	

	The following theorem shows that, in our current relative setting over a Noetherian domain $A$,  the theory of Gr\"obner degenerations \emph{generically over A} behaves just as the classical setting over a field. 
	As a direct consequence, we obtain an effective proof of Grothendieck's generic freeness lemma which is suitable for computations.

\begin{theorem}[Generic deformation to the initial module]
	\label{thm_gen_grob}
	Assume \autoref{setup_Grobner_deform}.
	Let $M \subset F$ be an $R$-submodule and suppose that $\iniTerm(M) = \left(a_1m_1,\ldots,a_bm_b\right)$ where $0\neq a_i \in A$ and $m_i \in F$ is a monomial.
   Choose $0 \neq a \in (a_1) \cap \cdots \cap (a_b) \subset A$.
	Then the following statements hold:
	\begin{enumerate}[\rm (i)]
		\item {\rm(}Generic Macaulay's theorem; {\rm cf.~\cite[Theorem 15.3]{EISEN_COMM}}{\rm)} $F/M \otimes_A A_a$ is a free $A_a$-module with a basis given by the set of monomials in $\left(F \otimes_A A_a\right) \setminus \left(m_1,\ldots,m_b\right)$.
		\item {\rm(}Generic flat degenerations; {\rm cf.~\cite[Theorem 15.17]{EISEN_COMM}}{\rm)} 
		There exist weight vectors $\omega = (\omega_1, \ldots, \omega_r)\in \ZZ_+^{r}$ and $\dd = (d_1, \ldots, d_\ell)\in \ZZ_+^{\ell}$ such that, over the graded polynomial ring $S = R[t]$ with $[S]_0 = A$, $\deg(x_i) = \omega_i$ and $\deg(t) = 1$, the graded $S$-submodule $E = \hom_{\omega,\dd}(M) \subset F[t] \cong \bigoplus_{i = 1}^\ell S(-d_i)$ satisfies the following properties:
		\begin{enumerate}[\rm (a)]
			\item $F[t]/E \otimes_{A[t]} A[t]/(t) \cong F/\iniTerm (M)$.
			\item $F[t]/E \otimes_{A[t]} A[t, t^{-1}] \cong F/M \otimes_{A} A[t, t^{-1}]$.
			\item $F[t]/E \otimes_A A_a$ is a free $A_a[t]$-module.			
		\end{enumerate}
	\end{enumerate}
\end{theorem}
\begin{proof}
	Throughout the proof we substitute $A$ by its localization $A_a$.
	Therefore we may assume that $\iniTerm(M) = (m_1,\ldots,m_b)$ and that $F/\iniTerm(M)$ is a free $A$-module with basis $B$, where $B$ is the set of monomials not in the initial module $\iniTerm(M)$.
	We choose a Gr\"obner basis $\{w_1,\ldots,w_b\}$ for $M$ such that $m_i = \iniTerm(w_i)$.
	
	(i) Let $w \in F \setminus M$.
	Choose the maximal term $m \in F$ of $w$ such that $m \in \iniTerm(M)$.
	Since $\iniTerm(M)$ is assumed to be an actual monomial submodule, we have $m = h m_i$ for some term $h \in R$ (cf.~\autoref{rem_basics_mons}(i)).
	Let $w' = w - hw_i$ and $m' \in F$ be the maximal term of $w'$ such that $m' \in \iniTerm(M)$.
	Notice that $[w] = [w'] \in F/M$ and that $m > m'$.
	Once again, by the descending chain condition of a monomial order, we can repeat this process finitely many times and obtain an element $z \in F$ whose terms only involve monomials in $B$ and such that $[w] = [z] \in F/M$.
	This shows that $B$ is a generating set of $F/M$ as an $A$-module.
	
	It is straightforward to check that the monomials in $B$ are $A$-linearly independent inside $F/M$.
	Indeed, if there is a dependence relation $w = \sum_{i=1}^k c_i\mu_i \in M$ with $0 \neq c_i \in A$ and $\mu_i \in B$, then we would obtain the contradiction $c_i\mu_i = \iniTerm(w) \in \iniTerm(M)$ for some $i$.
	
	(ii)
	Notice that there exists a weight vector $\omega \in \ZZ_+^{r}$ such that  $\omega \cdot \bn > \omega \cdot \bm$ for any pair $(\xx^\bn e_k, \xx^\bm e_k)$ in the following finite set
	$$
	\big\{ (\xx^\bn e_k, \xx^\bm e_k) \,\mid\, \text{$\xx^\bn e_k$ and  $\xx^\bm e_k$ are monomials of the same $w_i$ and  $\xx^\bn e_k > \xx^\bm e_k$}  \big\}.
	$$	
	Indeed, by utilizing Farkas' lemma, the proof follows as for the case of ideals (see \cite[Proposition 1.11]{STURMFELS_GROBNER}, \cite[Lemma 3.1.1]{HERZOG_HIBI_MONOMIALS}).
	Then, we can choose $\dd = (d_1, \ldots,d_\ell)$ with differences $d_k-d_{k+1}>0$ as large as needed so that we obtain $\iniTerm_>(w_i) = \iniTerm_{\omega,\dd}(w_i)$ for all $1 \le i \le b$.
	
	Let $E = \hom_{\omega,\dd}(M) \subset F[t]$.
	We shall prove that $E$ satisfies all the claimed properties.
	It is clear that property (a) holds. 
	Consider the automorphism $\varphi$ on $F[t,t^{-1}]$ given by $\varphi(\xx^\bn e_k) = t^{\deg_{\omega,\dd}(\xx^\bn e_k)}\xx^\bn e_k$, and notice that it takes  $E \otimes_{A[t]} A[t,t^{-1}]$ into $M \otimes_{A} A[t,t^{-1}]$.
	Therefore, $\varphi$ induces the isomorphism of property (b).
	Notice that for both properties (a) and (b) we do not need to substitute $A$ by its localization $A_a$.
	It remains to check property (c).
	
	We now define a monomial order $>'$ on $F[t]$ by setting that $\xx^\bn t^n e_k >' \xx^\bm t^m e_l$ if and only if 
	\begin{itemize}[--]
		\item $\deg_{\omega,\dd}(\xx^\bn e_k) + n> \deg_{\omega,\dd}(\xx^\bm e_l) + m$, or
		\item $\deg_{\omega,\dd}(\xx^\bn e_k) + n = \deg_{\omega,\dd}(\xx^\bm e_l) + m$ and $\xx^\alpha e_k > \xx^\beta e_l$.
	\end{itemize}
	By construction, for any $z \in E = \hom_{\omega,\dd}(M)$, we obtain that 
	$$
	\iniTerm_{>'}(z) \;=\;  t^{m} \iniTerm_>(w)
	$$
	for some $w \in M$ and $m \ge 0$.
	Therefore, since $\{w_1,\ldots,w_b\}$ is a Gr\"obner basis for $M \subset F$ with respect to $>$, it follows that $\{\hom_{\omega,\dd}(w_1),\ldots, \hom_{\omega,\dd}(w_b)\}$ is a Gr\"obner basis for $E \subset F[t]$ with respect to $>'$.
	Hence, we have $\iniTerm_{>'}(E) = \left(m_1, \ldots, m_b\right) \subset F[t]$.
	By utilizing part (i) applied to $E \subset F[t]$, the set of monomials in $F[t] \setminus \iniTerm_{>'}(E)$ provides an $A$-basis for $F[t] / E$, and so it follows that $F[t] / E$ is a free $A[t]$-module with a basis given by the monomials in $B$.	
	This settles property (c).
\end{proof}

	A direct consequence of the above theorem is the following celebrated result of Grothendieck \cite[Lemme 6.9.2]{EGA4_II}. 
	
	\begin{corollary}[Grothendieck's generic freeness lemma]\label{cor: gen free}
	Let $A$ be a Noetherian domain, $B$ be a finitely generated $A$-algebra, and $M$ be a finitely generated $B$-module. 
		Then there exists  a nonzero element $a \in A$ such that $M \otimes_{A} A_a$ is a free $A_a$-module.
	\end{corollary}
	\begin{proof}
		The result follows from \autoref{thm_gen_grob} after choosing a  surjection $R = A[x_1,\ldots,x_r] \twoheadrightarrow B$ from a polynomial ring and presenting $M$ as a quotient $M \cong F/N$ with $F$ a free $R$-module of finite rank.
	\end{proof}

	We also have the following graded version that provides generic freeness for all the graded parts of a graded module. 
	
	\begin{corollary}
		\label{cor_gen_free_grad}
		Let $A$ be a Noetherian domain, $B = \bigoplus_{i=0}^\infty B_i$ be a positively graded finitely generated $A$-algebra with $A$ acting on $B_0$, and $M$ a finitely generated graded $B$-module. 
		Then there exists  a nonzero element $a \in A$ such that $M_\nu \otimes_{A} A_a$ is a free $A_a$-module for all $\nu \in \ZZ$.
	\end{corollary}
  	\begin{proof}
		We may choose a graded surjection $R = A[x_1,\ldots,x_r] \surjects B$ with $R$ a positively graded polynomial ring over $A$, and a graded presentation $M \cong F/N$ with $F$ a graded free $R$-module.
		Putting aside the grading, \autoref{thm_gen_grob}(i) already implies that, after localizing at a nonzero element $a \in A$, $F/N$ becomes a free $A$-module with a basis given by the monomials not in an initial submodule.
		However, since monomials are clearly homogeneous with respect to any grading, the result of the corollary follows.
  	\end{proof}
	
	The same argument easily recovers a stronger result due to Hochster and Roberts \cite[Lemma~8.1]{Hochster_Roberts_Invariants}. 
	
	\begin{proposition}[Hochster -- Roberts]
		\label{prop_hochster_roberts}
Let $A$ be a Noetherian domain, $B$ be a finitely generated $A$-algebra, and
$C$ be a finitely generated $B$-algebra.
Let $E$ be a finitely generated $C$-module, $N \subset E$ be a finitely generated $A$-submodule, and $M \subset E$ be a finitely generated $B$-submodule.
	Set $D = E/(N+M)$.
	Then there exists  a nonzero element $a \in A$ such that $D \otimes_{A} A_a$ is a free $A_a$-module.
	\end{proposition}
	\begin{proof}
	We can choose polynomial rings $R = A[x_1,\ldots,x_s, x_{s+1}, \ldots, x_r]$ and $R' = A[x_1,\ldots,x_s] \subset R$
	so that we have surjections $R \twoheadrightarrow C$ and $R' \twoheadrightarrow B$. 
	We can also find a free $R$-module $F$ and finitely generated submodules $L_1 \subset F$, $L_2 \subset F$ and $L_3 \subset F$ over $A$, $R'$ and $R$, respectively, such that we have an isomorphism 
	$$
	D = E/\left(N + M\right) \; \cong \; F/\left(L_1+L_2+L_3\right).
	$$
	We continue using a monomial $>$ order on $F$ as in \autoref{setup_Grobner_deform}.
	Let $H = \iniTerm(L_1+L_2+L_3)\subset F$ be the $A$-module generated by $\iniTerm(w)$ for all $w \in L_1+L_2+L_3$.
	Similarly, we define $\iniTerm(L_1)$, $\iniTerm(L_2)$ and $\iniTerm(L_3)$ as modules over $A$, $R'$ and $R$, respectively.
	
	Let $G_1 \subset F$ be the $A$-submodule generated by the elements 
	$$
	w_1+w_2+w_3 \in F \;\text{ such that }\; w_1\in L_1, w_2\in L_2, w_3\in L_3 \;\text{ and }\; \iniTerm(w_1) \ge \iniTerm(w_2+w_3);
	$$
	notice that $G_1$ is a finitely generated $A$-module because the monomials appearing in $w_2+w_3$ are smaller or equal that some monomial appearing in an element $w_1 \in L_1$, and the list of monomials appearing in elements of $L_1$ is finite.
	Similarly, we define the $R'$-submodule $G_2 \subset F$ generated by the elements 
	$$
	w_2+w_3 \in F \;\text{ such that }\; w_2\in L_2, w_3\in L_3 \;\text{ and }\; \iniTerm(w_2) \ge \iniTerm(w_3);
	$$
	by a similar argument, $G_2$ is finitely generated $R'$-module. 
	
	By construction, we can see that $H = \iniTerm(G_1) + \iniTerm(G_2) + \iniTerm(L_3)$.
	We have that $\iniTerm(G_1)$, $\iniTerm(G_2)$ and $\iniTerm(L_3)$ are finitely generated modules over $A$, $R'$ and $R$, respectively, and thus we may write $$
	\iniTerm(G_1) = \sum_{i = 1}^h A \cdot  c_i \xx^{\alpha_i}, \quad \iniTerm(G_2) = \sum_{i = 1}^k R' \cdot  d_i \xx^{\beta_i} \quad \text{and} \quad \iniTerm(L_3) = \sum_{i = 1}^l R \cdot f_i \xx^{\gamma_i},
	$$ 
	where $c_i,d_i,f_i \in A$ and $\alpha_i, \beta_i, \gamma_i \in \NN^r$.
	Take a nonzero element $a  \in (c_1) \cap \cdots \cap (c_h) \cap (d_1) \cap \cdots \cap (d_k) \cap (f_1) \cap \cdots \cap (f_l) \subset A$.
	Finally, we can utilize the argument in the proof of \autoref{thm_gen_grob}(i) to show that 
	$$
	D \otimes_{A} A_a \; \cong \; F/\left(L_1+L_2+L_3\right) \otimes_{A} A_a
	$$
	is a free $A_a$-module with a basis given by the set of monomials not in $H$.
	\end{proof}	

	We now briefly discuss the extension of the Buchberger criterion into our current setting.
	As long as linear equations are solvable in $A$ (e.g., when $A$ is a PID or a polynomial ring  over a field) this leads to an algorithm for computing Gr\"obner bases (see \cite[\S 4.2]{ADAMS_LOUST_GROEBNER}).
	
	\begin{definition}
		Let $w \in F$ and $ \mathcal{G} = \{w_1,\ldots,w_k\} \in F$.
		We say that \emph{$w$ reduces to $z \in F$ modulo $\mathcal{G}$} if we can write 
		$$
		w = f_1w_1 + \cdots + f_kw_b + z
		$$
		for some elements $f_1,\ldots,f_b \in S$ such that $\iniTerm(f_iw_i) \le \iniTerm(w)$ and no nonzero term of $z$ belongs to the submodule $\left(\iniTerm(w_1),\ldots,\iniTerm(w_b)\right) \subset F$.
	\end{definition}
	
		\begin{theorem}[{cf.~\cite[Theorem 15.8]{EISEN_COMM},~\cite[Theorem 4.2.3]{ADAMS_LOUST_GROEBNER}}]
			\label{thm_Buchberger_crit}
		Let $M \subset F$ be an $R$-submodule, and $\mathcal{G} = \{w_1,\ldots,w_b\} \subset M$ be a generating set of $M$.
		Then $\mathcal{G}$ is a Gr\"obner basis for $M$ if and only if for all $(h_1,\ldots,h_b) \in R^b$ such that
		$$
		h_1\iniTerm(w_1) + \cdots+ h_b\iniTerm(w_b) \,=\,0 \quad \text{\rm(i.e., $(h_1,\ldots,h_b) \in \Syz(\iniTerm(w_1),\ldots,\iniTerm(w_b))$)}
		$$  
		and each $h_i$ is a term, we have that $h_1w_1+\cdots + h_bw_b$ reduces to zero modulo $\mathcal{G}$.
	\end{theorem}
	\begin{proof}
		If $\mathcal{G}$ is a Gr\"obner basis,  then it is clear that any such linear combination, as it is an element in $M$,  reduces to zero modulo $\mathcal{G}$.
		Hence we only need to prove the reverse implication, and we assume that $h_1w_1+\cdots + h_bw_b$ reduces to zero modulo $\mathcal{G}$ for any vector of terms $(h_1,\ldots,h_b) \in \Syz(\iniTerm(w_1),\ldots,\iniTerm(w_b))$.
		
		Choose a nonzero  $w \in M$ and write $w= f_1w_1 + \cdots + f_bw_b$.
		Let $m = \max\{\iniTerm(f_1w_1),\ldots,\iniTerm(f_bw_b)\}$ (i.e.,~the maximal respective monomial appearing among $\iniTerm(f_1w_1),\ldots,\iniTerm(f_bw_b)$).
		Notice that $m \ge \iniTerm(w)$, and that $\iniTerm(w) \in (\iniTerm(w_1),\ldots,\iniTerm(w_b))$ when  $m$ and $\iniTerm(w)$ have the same monomial.
		Therefore we assume that $m > \iniTerm(w)$.
		
		Let $\fJ = \{i \mid \text{ $m$ and $\iniTerm(f_iw_i)$ have the same monomial}\}$.
		We consider the vector $(h_1,\ldots,h_b) \in R^b$ such that  $h_i = \iniTerm(f_i)$ if $i \in \fJ$ and $h_i=0$ otherwise.
		Since $m > \iniTerm(w)$, we necessarily  have that $(h_1,\ldots,h_b) \in \Syz(\iniTerm(w_1),\ldots,\iniTerm(w_b))$.
		Then the hypothesis yields the existence of elements $g_1,\ldots,g_b \in R$ such that 
		$$
		h_1w_1 + \cdots + h_bw_b \;= \; g_1w_1 + \cdots + g_bw_b
		$$
		and $\iniTerm(h_1w_1 + \cdots + h_bw_b) \ge \iniTerm(g_iw_i)$.
		Set $f_i' = f_i-h_i+g_i$ and $m' = \max\{\iniTerm(f_1'w_1),\ldots,\iniTerm(f_b'w_b)\}$.
		By combining everything we obtain the equality
		$$
		w \;=\; f_1'w_1 + \cdots + f_b'w_b
		$$
		and the strict inequality $m > m'$.
		After finitely many reductions, we must obtain an expression $w = \gamma_1w_1 + \cdots + \gamma_bw_b$ with $\gamma_i \in R$ and such that $\iniTerm(w)$ and $\max\{\iniTerm(\gamma_1w_1),\ldots,\iniTerm(\gamma_bw_b)\}$ have the same monomial.
		This concludes the proof that $\mathcal{G}$ is a Gr\"obner basis for $M$.
	\end{proof}

	We now apply the above theorem to study the behavior of initial submodules under taking Frobenius powers in a positive characteristic setting. 
	If $A$ contains a field of characteristic $p > 0$ and $M \subset F$ is an $R$-submodule, we denote by $M^{[p]} \subset F$ the $R$-submodule generated by the $p$-th power $w^p = (f_1^p,\ldots,f_\ell^p)$ for all $w = (f_1,\ldots,f_\ell) \in M$.

	\begin{corollary}	
		\label{cor_Grob_Frob_pow}
		Suppose that $A$ contains a field of characteristic $p >0$.
		Let $M \subset F$ be an $R$-submodule and suppose that $\iniTerm(M) = \left(a_1m_1,\ldots,a_bm_b\right)$ where $0\neq a_i \in A$ and $m_i \in F$ is a monomial.
		Choose $0 \neq a \in (a_1) \cap \cdots \cap (a_b) \subset A$.
		Then, for all $e \ge 0$,  we have the equality $\iniTerm(M^{[p^e]} \otimes_{A} A_a) = \iniTerm(M)^{[p^e]}\otimes_{A} A_a$ and  that $F/M^{[p^e]} \otimes_{A} A_a$ is a free $A_a$-module.
	\end{corollary}
	\begin{proof}
		Since taking $p$-th powers commutes with localization, we may substitute $A$ by $A_a$.
		Thus we assume that $\iniTerm(M) = (m_1,\ldots,m_b) \subset F$ is a monomial submodule.
		Choose $w_1,\ldots,w_b \in M$ with $m_i = \iniTerm(w_i)$, and set $\mathcal{G} = \{w_1,\ldots,w_b\}$.
		We have that $\Syz(m_1,\ldots,m_b)$ is generated by  the divided Koszul relations of the $m_i$'s; the same proof of \cite[Lemma 15.1]{EISEN_COMM} applies in our case.
		As these relations commute with taking $p$-th powers, we have that $\Syz(m_1^p,\ldots,m_b^p) = \Syz(m_1,\ldots,m_b)^{[p]}$. 
		By \autoref{thm_Buchberger_crit}, for any vector of terms $(h_1,\ldots,h_b) \in \Syz(m_1,\ldots,m_b)$ we have that $h_1w_1+\cdots + h_bw_b$ reduces to zero modulo $\mathcal{G}$, and in particular this implies that $h_1^pw_1^p+\cdots+h_b^pw_b^p$ reduces to zero modulo $\mathcal{G}^p = \{w_1^p,\ldots,w_b^p\}$.
		But then \autoref{thm_Buchberger_crit} implies that $\mathcal{G}^p$ is a Gr\"obner basis for $M^{[p]}$, and this gives the claimed equality $\iniTerm(M^{[p]}) = \iniTerm(M)^{[p]}$.
		
		By induction, we obtain that $\iniTerm(M^{[p^e]}) = \iniTerm(M)^{[p^e]}$ for all $e \ge 0$.
		Therefore \autoref{thm_gen_grob} yields that $F/M^{[p^e]}$ is a free $A$-module for all $e \ge 0$.
	\end{proof}

	\section{Generic square-free Gr\"obner degenerations}
	\label{sect_sqr_grob}
	
	In this section, we also provide a specific generic freeness result for local cohomology modules under the existence of a square-free Gr\"obner degeneration.
	Throughout this section we use the following specialization of  \autoref{setup_Grobner_deform}.

	\begin{setup}
		\label{setup_sqfree_Grob}
		Let $A$ be a Noetherian domain,  $R =  A[x_1,\ldots,x_r]$ be a positively graded polynomial ring over $A$, and $\mm = [R]_+= (x_1,\ldots,x_r)$ be the graded irrelevant ideal.
		As in \autoref{setup_Grobner_deform}, let $>$ be a monomial order on $R$.
	\end{setup}

	The main goal of this section is to prove the following theorem. 
	It is inspired by the work of Conca and Varbaro \cite{CONCA_VARBARO} on square-free Gr\"obner degenerations.

	\begin{theorem}
		\label{thm_sqr_free_deg}
		Assume \autoref{setup_sqfree_Grob} and suppose that $A$ contains a field $\kk$.
		Let $I \subset R$ be a homogeneous ideal and
		$\iniTerm(I) = \big(a_1\xx^{\beta_1}, \,\ldots\,, a_b\xx^{\beta_b} \big)$ be its corresponding initial ideal, where $0 \neq a_i \in A$ and $\beta_i \in \NN^r$.
		Suppose that each monomial $\xx^{\beta_i}$ is square-free.
		Choose $0 \neq a \in (a_1) \cap \cdots \cap (a_b) \subset A$.
		Then $\HL^i(R/I) \otimes_{A} A_a$ is $A_a$-free for all $i \ge 0$.
	\end{theorem}

	We point out that the theorem is sharp in the following sense.
	
	\begin{remark}[{\cite[Remark 3]{REISNER_CM}}]
		\label{rem_triang_PP_2}
		The condition that $A$ contains a field cannot be avoided in \autoref{thm_sqr_free_deg}.
		Let $R = \ZZ[x_1,\ldots,x_6]$ and take the square-free monomial ideal 
		$$
		I \;=\; \left(x_1x_2x_3, \,x_1x_2x_4,\, x_1x_3x_5,\, x_1x_4x_6,\, x_1x_5x_6, \, x_2x_3x_6,\, x_2x_4x_5,\, x_2x_5x_6,\, x_3x_4x_5,\, x_3x_4x_6\right)
		$$
		corresponding to the minimal triangulation of the projective plane.
		By \cite[Remark 3]{REISNER_CM}, for a field $\KK$, we have that $R/I \otimes_\ZZ \KK$ is Cohen-Macaulay if and only if $\text{char}(\KK) \neq 2$.
		Therefore, it cannot happen that $\HL^i(R/I)$ is $\ZZ$-flat for all $i \ge 0$.
	\end{remark}
	
	We  state the following result that will be useful for us.
	
	\begin{theorem}[{\cite[Theorem A]{FIBER_FULL}}]
		\label{thm_fib_full_mod}
		Let $B$ be a Noetherian ring and $S$ be a positively graded finitely generated $B$-algebra.
		Let $\MM = [S]_+$ be the graded irrelevant ideal of $S$.
		Let $M$ be a finitely generated graded $S$-module and suppose that $M$ is flat over $B$. 
		Then the following conditions are equivalent: 
		\begin{enumerate}[\rm (i)]
			\item $\HH_{\MM}^i(M)$ is a flat $B$-module for all $i \ge 0$.
			\item $\HH_{\MM}^i\left(M \otimes_{B} B_\pp/\pp^qB_\pp\right)$ is $(B_\pp/\pp^qB_\pp)$-flat for all $i \ge 0$,  $q \ge 1$ and $\pp \in \Spec(B)$.
			\item The natural map $\HH_{\MM}^i\left(M \otimes_{B} B_\pp/\pp^qB_\pp\right) \rightarrow \HH_{\MM}^i\left(M \otimes_{B} B_\pp/\pp B_\pp\right)$ is surjective for all $i \ge 0$, $q \ge 1$ and $\pp \in \Spec(B)$.
		\end{enumerate}
		Moreover, when any of the above equivalent conditions is satisfied, we have the following base change isomorphism
		$$
		\HH_{\MM}^i(M) \otimes_{B} C \;\xrightarrow{\cong}\; \HH_{\MM}^i(M \otimes_{B} C)
		$$
		for all $i \ge 0$ and any $B$-algebra $C$.
	\end{theorem}

		\begin{definition}
		Under the notation of \autoref{thm_fib_full_mod}, we say that  a finitely generated graded $S$-module is \emph{fiber-full over B} if $M$ is $B$-flat and $\HH_{\MM}^i(M)$ is $B$-flat for all $i \ge 0$.
	\end{definition}

	Condition (iii) of the above theorem is a relaxation of the closely related notions of \emph{algebras having liftable local cohomology} introduced by Koll\'ar and Kov\'acs \cite{KOLLAR_KOVACS} and \emph{cohomologically full rings} introduced by Dao, De Stefani and Ma \cite{COHOM_FULL_RINGS}.
	The term \emph{fiber-full} was coined by Varbaro in \cite[Definition 3.8]{CONFERENCE_LEVICO}.
	Further developments were made with the construction of the \emph{fiber-full scheme} \cite{FIB_FULL_SCHEME, cid2022local}.
	
In order to state the following lemma, we recall that 
a positively graded ring $R$ over a field 
is \emph{cohomologically full} if for any equicharateristic local ring $T$ a surjection 
$\phi \colon (T, \mathfrak{n}) \twoheadrightarrow R_\MM$
which induces an isomorphism $T/\sqrt{0} T \cong R_\MM/\sqrt{0}R_\MM$ must also induce surjective natural maps $\HH_{\mathfrak{n}}^i (T) \twoheadrightarrow \HH^i_{\MM} (R)$ for all $i$.
	
	\begin{lemma}[{\cite[Proposition 2.3, Proposition 3.3]{CONCA_VARBARO}}]
		\label{lem_sqr_free}
		Let $S$ be a positively graded polynomial ring over a field and $>$ be a monomial order on $S$.
		Let $\aaa \subset S$ be a homogeneous ideal.
		If $\iniTerm(\aaa)$ is a square-free monomial ideal, then both  $S/\aaa$ and $S/\iniTerm(\aaa)$ are cohomologically full rings.
	\end{lemma}
	
	After recalling these needed results we are ready for the proof \autoref{thm_sqr_free_deg}.
	
	\begin{proof}[Proof of \autoref{thm_sqr_free_deg}]
		Let $\delta_i = \deg(x_i) > 0$.
		To simplify the notation, we substitute $A$ by $A_a$ (where $0 \neq a \in A$ is the chosen element), hence we assume that $\iniTerm(I) = (\xx^{\beta_1},\,\ldots\,, \xx^{\beta_b}) \subset R$. 
		By \autoref{thm_gen_grob}, $R/I$ is $A$-free and there exists an ideal $J \subset S=R[t]=A[t, x_1, \ldots, x_r]$ such that the following three conditions are satisfied: (a) $S/J \otimes_{A[t]} A[t]/(t) \cong R/\iniTerm(I)$, (b) $S/J \otimes_{A[t]} A[t, t^{-1}] \cong R/I \otimes_{A} A[t,t^{-1}]$, and (c) $S/J$ is a free $A[t]$-module.
		Furthermore, we can assume that $J \subset S$ is bihomogeneous and $S$ is bigraded with $\bideg(x_i) = (\delta_i, \omega_i)$ and $\bideg(t) = (0, 1)$, where $\omega = (\omega_1,\ldots,\omega_r)$ is the weight vector from \autoref{thm_gen_grob}.
		
		Let $B = A[t]$.
		We first show that $S/J$ is fiber-full over $B$.
		Let $P \in \Spec(B)$, $\pp = P \cap A \in \Spec(A)$, and $\bFF =  A_\pp/\pp A_\pp$.
		Set $\overline{B} = B \otimes_{A} \bFF=\bFF[t]$ and $\overline{R} = R \otimes_A \bFF = \bFF[x_1,\ldots,x_r]$, and consider the prime ideal $\overline{P} = P \overline{B} \in \Spec(\overline{B})$.
		We may assume that $>$ is also a monomial order on $\overline{R}$.
		Let $\aaa = I \overline{R} \subset \overline{R}$ and  $\bbb = \iniTerm(I)\overline{R} = (\xx^{\beta_1}, \ldots, \xx^{\beta_b}) \subset \overline{R}$.
		It is clear that $\bbb \subseteq \iniTerm(\aaa)$.
		From the three conditions (a), (b), (c) that $S/J$ satisfies, it follows that $\overline{R}/\aaa$ and $\overline{R}/\bbb$ have the same Hilbert function.
		But since the Hilbert function of $\overline{R}/\aaa$ also coincides with the one of $\overline{R}/\iniTerm(\aaa)$, we obtain the equality $\iniTerm(\aaa) = \bbb = (\xx^{\beta_1}, \ldots, \xx^{\beta_b}) \subset \overline{R}$.
		
		Let $\nnn  = P + \mm \subset R$.
		We analyze the following two cases: 
		\begin{enumerate}
			\item Suppose that $t \in \overline{P}$ (i.e., $\overline{P} = (t)$).
			In this case, it follows that $S/J \otimes_{B} B_P/PB_P \cong \overline{R}/\iniTerm(\aaa)$.
			Due to \autoref{lem_sqr_free}, we obtain the natural surjection 
			$
			\HH_{\nnn}^i(S/J \otimes_{B} B_P/P^qB_P)	 \surjects \HH_{\nnn}^i(S/J \otimes_{B} B_P/PB_P) \cong \HH_{\nnn}^i(\overline{R}/\iniTerm(\aaa))
			$			
			for all $i \ge 0, q \ge 1$.
			\item Suppose that $t \not\in \overline{P}$.
			Then $S/J \otimes_{B} B_P/PB_P \cong \overline{R}/\aaa$.
			Similarly, \autoref{lem_sqr_free} yields the natural surjection
			$
			\HH_{\nnn}^i(S/J \otimes_{B} B_P/P^qB_P)	 \surjects \HH_{\nnn}^i(S/J \otimes_{B} B_P/PB_P) \cong \HH_{\nnn}^i(\overline{R}/\aaa)
			$			
			for all $i \ge 0, q \ge 1$.
		\end{enumerate}
		Therefore, for all $i \ge 0, q \ge 1, P\in \Spec(B)$, since 
		$
		\HH_{\nnn}^i(S/J \otimes_{B} B_P/P^qB_P)	 \cong \HH_{\mm}^i(S/J \otimes_{B} B_P/P^qB_P),	
		$
		we also get the surjection 
		$$
		\HH_{\mm}^i(S/J \otimes_{B} B_P/P^qB_P)	 \;\surjects\; \HH_{\mm}^i(S/J \otimes_{B} B_P/PB_P).
		$$
		By \autoref{thm_fib_full_mod}, we obtain that $S/J$ is fiber-full over $B$.
		We will now prove that each $\HL^i(S/J)$ is actually $B$-free, and not just $B$-flat.
		If we show that $\HL^i(S/J)$ is $B$-free, by the arbitrary base change property of fiber-full modules (see \autoref{thm_fib_full_mod}), we obtain that 
		$$
		\HL^i(R/I) \;\cong\;  \HL^i\left(S/J \otimes_B B/(t-1)\right) \;\cong\;  \HL^i(S/J) \otimes_B B/(t-1)
		$$
		is free over $A \cong B/(t-1)$.
		
		Under the bigrading of $S$ introduced above,   $\left[\HL^i(S/J)\right]_{(\mu,*)} = \bigoplus_{\nu \in \ZZ} \left[\HL^i(S/J)\right]_{(\mu,\nu)}$ is a finitely generated  $B$-module for all $\mu \in \ZZ$.
		Hence we obtain that $\HL^i(S/J)$ is a projective $B$-module.
		When $\HL^i(S/J)$ is not finitely generated as a $B$-module, a classical result of Bass \cite[Corollary 4.5]{BASS_PROJ} already implies that  it is $B$-free.

 		From the arbitrary base change property of fiber-full modules, we deduce that the long exact sequence in cohomology induced by $0 \rightarrow S/J(0,-1) \xrightarrow{\cdot t} S/J \rightarrow R/\iniTerm(I) \rightarrow 0$ splits into the following short exact sequences of bigraded $S$-modules
		$$
		0 \;\rightarrow\; \HL^i(S/J)(0,-1) \; \xrightarrow{\cdot t} \; \HL^i(S/J) \rightarrow\; \HL^i(R/\iniTerm(I))\; \rightarrow \; 0
		$$
		for $i \ge 0$.
		Fix $\mu \in \ZZ$ and $i \ge 0$. 
		Let $M := \big[\HL^i(S/J)\big]_{(\mu,*)}$ and $N := \big[\HL^i(R/\iniTerm(I))\big]_{(\mu,*)}$.
		We see $B=A[t]$ as a standard graded polynomial ring.
		Thus we have a short exact sequence $0 \rightarrow M(-1) \xrightarrow{\cdot t} M \rightarrow N \rightarrow 0$ of finitely generated graded $B$-modules.
		
		Let $R_0 = \kk[x_1,\ldots,x_r] \subset R$ and $\bbb = (\xx^{\beta_1}, \ldots,\xx^{\beta_b}) \subset R_0$, and denote also by $\mm$ the irrelevant ideal of $R_0$.
		Since $\HL^i(R/\iniTerm(I))\cong \HL^i(R_0/\bbb) \otimes_\kk A$, it follows that $N$ is a free $A$-module.
		Hence we obtain a free $A$-module $W \subset M$ which is isomorphic to $N$ and the equality of $A$-modules
		$
		M = W \oplus t\,M. 
		$
		By Nakayama's lemma,  $M$ is generated as a $B$-module by an $A$-basis of $W$.
		One can check that such generating set is actually a $B$-basis of $M$.
		Therefore, the proof of the theorem is complete.
	\end{proof}

  \section{Generic freeness for certain filtered algebras}
  \label{sect_gen_free_filtered}
  
  In this short section, we discuss the generic freeness of modules over certain (not necessarily commutative) filtered rings.
  The result proven here will play an important role in our study of local cohomology modules in a characteristic zero setting.
  We have the following theorem that follows from \autoref{cor_gen_free_grad}.
     
\begin{theorem}\label{thm: noncom gen freeness}
Let $D = \bigcup_{i \geq 0} D_i$  be a {\rm(}not necessarily commutative{\rm)} filtered algebra over a commutative ring $R = D_0$ 
such that the associated graded ring $\gr (D) \coloneqq \bigoplus_{i \geq 0} D_i/D_{i - 1}$ {\rm(}with $D_{-1} = 0${\rm)} is a finitely generated commutative $R$-algebra.  
Suppose that $R$ is a finitely generated algebra over a Noetherian domain $A$. 
Then, for any finitely generated left $D$-module $M$,  there exists an element $0 \neq a \in A$ such that $A_a \otimes_A M$ is a free $A_a$-module.
\end{theorem}
\begin{proof}
Fix a system of generators $f_1, \ldots, f_n$ of $M$ and let $M_i = D_i \langle f_1, \ldots, f_n \rangle$.
By  construction, $\gr_D(M) \coloneqq \bigoplus_{i = 0}^\infty M_i/M_{i -1}$ is 
generated over $\gr(D)$ by $f_1, \ldots, f_n \in M_0$.
Since  $\gr (D)$ is a finitely generated $A$-algebra, \autoref{cor_gen_free_grad} gives an element $0 \neq a \in A$ such that the graded pieces of $\gr_D (M) \otimes_{A} A_a$
are free $A_a$-modules. 
Thus, for any $i \ge 0$ there is a short exact sequence of $A_a$-modules
\[
0 \to A_a \otimes_{A} M_{i - 1} \to A_a \otimes_{A}  M_i \to A_a \otimes_{A}  M_{i}/M_{i - 1} \to 0,
\]
as $A_a \otimes_{A}  M_{i}/M_{i - 1}$ is free, the sequence splits. 
Moreover, $M_i = 0$ for all $i < 0$.  It follows by induction that all $A_a \otimes_{A} M_i$ and
$A_a \otimes_{A} M$ are $A_a$-free.  
\end{proof}
  
We will apply this theorem to the subring of the ring of differential operators generated by derivations as the necessary properties are provided by
\autoref{prop_derivation_ring}.
  For completeness, we recall the result of Artin--Small--Zhang \cite{ASZ} proving \emph{generic flatness} over certain \emph{strongly Noetherian algebras}.
  This generic flatness result covers more general non-commutative algebras, but it comes with the price of stronger conditions over the coefficient ring.
  
  \begin{definition}
		Let $D$ be a right Noetherian algebra over a commutative Noetherian ring $A$.
		We say that $D$ is \emph{strongly right Noetherian over $A$} if $D \otimes_{A} B$ is right Noetherian for any commutative Noetherian $A$-algebra $B$.
  \end{definition}

	In \autoref{thm_gen_prop_smooth_A}(ii), we show that the ring of differential operators is strongly right Noetherian whenever we have smoothness over a coefficient ring that is Noetherian and contains the field of rational numbers.
In the commutative case, algebras \emph{essentially} of finite type are strongly Noetherian \cite[Proposition~4.1]{ASZ}.

	\begin{theorem}[Artin--Small--Zhang \cite{ASZ}]
		Let $A$ be a domain of finite type over a field or an excellent Dedekind domain.
		Assume $D$ is a strongly right Noetherian algebra over $A$.
		Then, for any finitely generated right $D$-module $M$, there is an element $0 \neq a \in A$ such that $M \otimes_{A} A_a$ is a flat $A_a$-module.
	\end{theorem}

	\section{Differential operators and smooth algebras over a coefficient ring}
	\label{sect_diff_ops_smooth}
	
	Here we study several properties of differential operators in smooth algebras over a coefficient ring that is Noetherian and contains the field of rational numbers. 
	Our main result is presented in \autoref{thm_gen_prop_smooth_A}.
	A general and complete reference on the topic of differential operators is \cite[\S 16]{EGAIV_IV}.
	Throughout this section the following setup is in place.
	
	\begin{setup}
		Let $R$ be a commutative ring and $A$ be a subring.
	\end{setup}
	
	For two $R$-modules $M$ and $N$, we regard $\Hom_A(M, N)$ as an $(R\otimes_A R)$-module, by setting 
	$$
	\left((r \otimes_A s) \delta\right)(w) \,=\, r \delta(sw) \quad \text{ for all } \delta \in \Hom_A(M, N), \; w \in M,\; r,s \in R. 
	$$ 
	We use the bracket notation $[\delta,r](w) \coloneqq \delta(rw)-r\delta(w)$ for $\delta \in \Hom_A(M, N)$, $r \in R$ and $w \in M$.
	Unless specified otherwise, whenever we consider an $(R \otimes_A R)$-module as an $R$-module, we do 
	so by letting $R$ act via the left factor of $R \otimes_A R$. 
	Differential operators are defined inductively as follows.
	
	\begin{definition}
		\label{def_diff_ops}
		Let $M, N$ be two $R$-modules.
		The \textit{$m$-th order $A$-linear differential operators}, denoted as
		$\,
		\Diff_{R/A}^m(M, N) \subseteq \Hom_A(M, N)$,
		form an $(R\otimes_A R)$-module that is defined inductively~by
		\begin{enumerate}[\rm (i)]
			\item $\Diff_{R/A}^{0}(M,N) \coloneqq \Hom_R(M,N)$.
			\item $\Diff_{R/A}^{m}(M, N) \coloneqq
			\big\lbrace \delta \in \Hom_A(M,N) \,\mid\, [\delta, r] \in \Diff_{R/A}^{m-1}(M, N) 
			\,\text{ for all }\, r \in R \big\rbrace$.
		\end{enumerate}
		The set of all \textit{$A$-linear differential operators from $M$ to $N$} 
		is the $(R \otimes_A R)$-module
		$
		\Diff_{R/A}(M, N) \coloneqq \bigcup_{m=0}^\infty \Diff_{R/A}^m(M,N).
		$
		To simplify notation, one sets $D_{R/A}^m \coloneqq \Diff_{R/A}^m(R, R)$ and $D_{R/A} \coloneqq \Diff_{R/A}(R, R)$.
		One says that $D_{R/A}$ is the \emph{ring of $A$-linear differential operators of $R$} (which is not necessarily a commutative ring).
		Following the notation of \cite[Chapter 15]{McCONNELL_ROBSON}, we denote by $\Delta(R/A) \subset D_{R/A}$ the $A$-subalgebra generated by $D_{R/A}^1 = R \oplus \Der_{R/A}$, and we call it the \emph{derivation ring of $R$ over $A$}.
	\end{definition}

	\begin{remark}
		\label{rem_relative_Weyl_alg}
		Let $A$ be a ring containing the field $\QQ$ of rational numbers and $R = A[x_1,\ldots,x_r]$ be a polynomial ring.
		Then $D_{R/A}$ coincides with the \emph{relative Weyl algebra}
		$
		A[x_1,\ldots,x_r]\langle \partial_1, \ldots, \partial_r \rangle,
		$
		which is a quotient of the free $A$-algebra generated by $x_1,\ldots,x_r,\partial_1,\ldots,\partial_r$ modulo the two-sided ideal generated by the relations $x_ix_j = x_jx_i$, $\partial_i\partial_j = \partial_j\partial_i$ and $\partial_ix_j = x_j\partial_i + \delta_{i,j}$; here $\delta_{i,j}$ denotes Kronecker's symbol. 
		See, e.g., \cite[Th\'eor\`eme 16.11.2]{EGAIV_IV}, \cite[Example 4]{OBERST_NOETH_OPS}.
	\end{remark}

	To describe differential operators, one uses the module of principal parts.
	Consider the multiplication map 
	$
	\mu_{R/A} \colon R \otimes_A R \rightarrow R, \; 	r \otimes_A s \mapsto rs.
	$
	The kernel of this map is the diagonal ideal $\Delta_{R/A} \coloneqq \Ker(\mu_{R/A})\subset R \otimes_A R$. 
	\begin{definition}
		For an $R$-module $M$, the {\em module of $m$-th principal parts} is defined as	
		$
		P_{R/A}^m(M) \coloneqq \frac{R \otimes_A M}{	\Delta_{R/A}^{m+1}  \left(R \otimes_A M\right)}.
		$
		This is a module over $R \otimes_A R$ and thus over $R$.
		For simplicity of notation, we set $ P_{R/A}^m \coloneqq P_{R/A}^m(R)$.
	\end{definition}

	\begin{remark}
		\label{rem_finite_gen_princ}
		If $A$ is a Noetherian and $R$ is essentially of finite type over $A$, then $P_{R/A}^m$ is a finitely generated $R$-module for all $m \ge 0$ (see, e.g., \cite[Remark 3.3]{NOETH_OPS}).
	\end{remark}

	By an abuse of notation, we also denote by $\mu_{R/A}$ the natural multiplication map $P_{R/A}^m \rightarrow R$.
	For any $R$-module $M$, we consider the universal map 
	$\, d_{R/A}^m \colon M \rightarrow P_{R/A}^m(M), \,\, w  \mapsto \overline{1 \otimes_A w} $.
	The following result is a fundamental characterization of the modules of differential operators.
	
	\begin{proposition}[{\cite[Proposition 16.8.4]{EGAIV_IV}, \cite[Theorem 2.2.6]{AFFINE_HOPF_I}}]
		\label{prop_represen_diff_opp}
		Let $M$ and $N$ be $R$-modules and let $m\ge 0$.
		Then the following map is an isomorphism of $R$-modules:
		$$
			{\big(d_{R/A}^m\big)}^*\, \colon \, \Hom_R\big(P_{R/A}^m(M), N\big) 
			\,\,\xrightarrow{\cong} \,\,\Diff_{R/A}^m(M, N), \quad
			\psi   \,\,\mapsto \, \psi \circ d_{R/A}^m.
			$$
	\end{proposition}

		We now point out some of the advantages that working with the derivation ring $\Delta(R/A)$ provides.
	In this section, we consider $\Delta(R/A)$ as a filtered ring in terms of the order of the differential operators. 
	That is, we consider the filtration where $F_m(\Delta(R/A))$ is the $R$-module generated by all the products of a most $m$ derivations in $\Der_{R/A}$.
	The corresponding associated graded ring is given by 
	$$
	\gr(\Delta(R/A)) \coloneqq \bigoplus_{m=0}^\infty\, \frac{F_m(\Delta(R/A))}{F_{m-1}(\Delta(R/A))}.
	$$
	\begin{lemma}
		\label{prop_derivation_ring}
		The following statements hold: 
		\begin{enumerate}[\rm (i)]
			\item $\gr(\Delta(R/A))$ is commutative and there is a canonical surjection $\Sym_R(\Der_{R/A}) \surjects \gr(\Delta(R/A))$.
			\item If $A$ is Noetherian and $R$ is a essentially of finite type over $A$, then $\gr(\Delta(R/A))$ and $\Delta(R/A)$ are Noetherian rings.
		\end{enumerate}
	\end{lemma}
	\begin{proof}
		Part (i) follows from \cite[Proposition 15.1.19]{McCONNELL_ROBSON}.
		Notice that, for all $\delta_1, \delta_2 \in \Der_{R/A}$, $r,s \in R$, we have 
		$
		[\delta_1, \delta_2](rs) = r[\delta_1, \delta_2](s) + s[\delta_1, \delta_2](r).
		$
		Hence $[\delta_1, \delta_2] = \delta_1\delta_2 - \delta_2\delta_1 \in \Der_{R/A}$.

		Since $R$ is essentially of finite type over $A$, $\Omega_{R/A}$ is a finitely generated $R$-module (see \autoref{rem_finite_gen_princ}), and so $\Der_{R/A} = \Hom_R(\Omega_{R/A}, R)$ also is.
		Therefore, by part (i), we obtain that $\gr(\Delta(R/A))$ and $\Delta(R/A)$ are Noetherian rings (see \cite[Theorem 15.1.20]{McCONNELL_ROBSON}).
		So the proof of part (ii) is complete.
	\end{proof}

	We say that $A \rightarrow R$ is a \emph{smooth} ring map if $R$ is  \emph{formally smooth} and of finite presentation over $A$.
	Similarly, the ring map $A \rightarrow R$ is \emph{\'etale} if $R$ is  \emph{formally \'etale} and of finite presentation over $A$.
	For more details on these notions, see
	\cite[\href{https://stacks.math.columbia.edu/tag/00TH}{Tag 00TH}]{stacks-project} and \cite[\href{https://stacks.math.columbia.edu/tag/00UP}{Tag 00UP}]{stacks-project}.
	We cover a result of M\'asson \cite{MASSON} regarding the behavior of differential operators under (formally) \'etale ring maps (for completeness, we include a short account working over an arbitrary coefficient ring $A$).
	
	We first point out the following base change property of differential operators under a smooth setting.
	
	\begin{lemma}
		\label{lem_base_change_diff}
		Suppose that $A \rightarrow R$ is formally smooth and  $P_{R/A}^m$ is a finitely generated $R$-module for all $m \ge 0$ .
		Then, for any $A$-algebra $B$, we have a base change isomorphism $D_{R/A} \otimes_{A} B \xrightarrow{\cong} D_{(R\otimes_{A} B)/B}$. 
	\end{lemma}
	\begin{proof}
		Due to \cite[Proposition 16.10.2]{EGAIV_IV} and \cite[D\'efinition 16.10.1]{EGAIV_IV}, we have that $\Princ^m$ is a projective $R$-module for all $m\ge0$.
		Furthermore, $R$ is $A$-flat by \cite[Th\'eor\`eme 19.7.1]{EGAIV_I}.
		Let $R_B = R \otimes_{A} B$.
		The short exact sequence $0 \rightarrow \Delta_{R/A} \rightarrow R\otimes_{A} R \xrightarrow{\mu_{R/A}}R\rightarrow 0$ induces a short exact sequence $0 \rightarrow \Delta_{R/A} \otimes_{A} B\rightarrow R_B\otimes_{B} R_B \xrightarrow{\mu_{R_B/B}}R_B\rightarrow 0$, and so we obtain that $\Delta_{R/A} \otimes_{A} B = \Delta_{R_B/B}$.
		Similarly, by considering the short exact sequence $0 \rightarrow \Delta_{R/A}^{m+1} \rightarrow R\otimes_{A} R \rightarrow P_{R/A}^m \rightarrow 0$ and taking the tensor product $-\otimes_{A} B$, we obtain the isomorphism $P_{R/A}^m \otimes_{A} B \cong P_{R_B/B}^m$.
		Thus \autoref{prop_represen_diff_opp} yields the natural isomorphisms
		$$
		D_{R/A}^m \otimes_{A} B \;\cong\; \Hom_R(P_{R/A}^m, R) \otimes_{A} B \;\cong\; \Hom_{R_B}(P_{R/A}^m \otimes_{A} B, R_B) \;\cong\; D_{R_B/B}^m.
		$$
		So the result follows.
	\end{proof}
	
	\begin{remark}
		\label{rem_existence_dhat_map}
		Let $\varphi\colon R \rightarrow T$ be a formally \'etale ring map of $A$-algebras.
		The kernel of the natural multiplication map $T \otimes_R P_{R/A}^m \rightarrow T \otimes_R R  \cong T$ is a nilpotent ideal. 
		Therefore, from the definition of formally \'etale algebras, we obtain a unique $A$-algebra homomorphism $\widehat{d}_{T/A}^m \colon T \rightarrow  T \otimes_R P_{R/A}^m$ that makes the following diagram commute
		\begin{center}
			\begin{tikzpicture}[baseline=(current  bounding  box.center)]
				\matrix (m) [matrix of math nodes,row sep=5em,column sep=7em,minimum width=2em, text height=1.3ex, text depth=0.25ex]
				{
					T & T \\
					R & T \otimes_R P_{R/A}^m. \\
				};
				\path[-stealth]
				(m-2-1) edge node [left] {$\varphi$} (m-1-1)
				(m-1-1) edge node [above] {$1$} (m-1-2)
				(m-2-1) edge node [above] {$1 \otimes_R d_{R/A}^m$} (m-2-2)
				(m-2-2)  edge  node [right] {$T \otimes_R \mu_{R/A}$} (m-1-2)
				;	
				\draw[->,dashed] (m-1-1)--(m-2-2) node [midway,above] {$\widehat{d}_{T/A}^m$};	
			\end{tikzpicture}	
		\end{center}
		From this commutative diagram we get that $\Delta_{T/A}^{m+1} \,\cdot\, \widehat{d}_{T/A}^m = 0$. 
		Indeed,  $\Ker(T\otimes_R \mu_{R/A})^{m+1} = 0$ and $t\,\widehat{d}_{T/A}^m(1) - \widehat{d}_{T/A}^m(t) \in \Ker(T\otimes_R \mu_{R/A})$ for all $t \in T$.
		This implies that $\widehat{d}_{T/A}^m$ is a differential operator of order at most $m$ (see \cite[Proposition 2.2.3]{AFFINE_HOPF_I}). 
	\end{remark}

	\begin{proposition}
		\label{prop_isom_princ_part_etale}
		Let $\varphi \colon R \rightarrow T$ be a formally \'etale ring map of $A$-algebras.
		Then there is a unique $T$-module isomorphism $\phi^m \colon T \otimes_R P_{R/A}^m \xrightarrow{\,\cong\,} P_{T/A}^m$ making the following diagram commute
		\begin{center}
			\begin{tikzpicture}[baseline=(current  bounding  box.center)]
				\matrix (m) [matrix of math nodes,row sep=5em,column sep=7em,minimum width=2em, text height=1.3ex, text depth=0.25ex]
				{
					T & T \otimes_R P_{R/A}^m \\
				        & P_{T/A}^m. \\
				};
				\path[-stealth]
				(m-1-1) edge node [above] {$d_{T/A}^m$} (m-2-2)
				(m-1-1) edge node [above] {$\widehat{d}_{T/A}^m$} (m-1-2)
				;			
				\draw[->,dashed] (m-1-2)--(m-2-2) node [midway,right] {$\phi^m$};	
			\end{tikzpicture}	
		\end{center}
	\end{proposition}
	\begin{proof}
		From \autoref{rem_existence_dhat_map} and \autoref{prop_represen_diff_opp}, there is a unique $T$-linear map $\psi^m \colon P_{T/A}^m \rightarrow T \otimes_R P_{R/A}^m$ such that 
		$$
		\widehat{d}_{T/A}^m \;=\; \psi^m  \circ d_{T/A}^{m}.
		$$ 
		To conclude the proof, we construct  the $T$-linear map $\phi^m \colon T \otimes_R P_{R/A}^m \rightarrow P_{T/A}^m$ as the explicit inverse of $\psi^m$.
		There is a natural induced map $P^m_\varphi \colon P_{R/A}^m \rightarrow P_{T/A}^m,\, \overline{r_1 \otimes_A r_2} \mapsto \overline{\varphi(r_1) \otimes_A \varphi(r_2)}$.
		We then set 
		$$
		\phi^m \colon T \otimes_R P_{R/A}^m \rightarrow P_{T/A}^m, \quad t \otimes w \mapsto tP_{\varphi}^m(w).
		$$
		Notice that we have the following commutative diagram 
		\begin{center}
			\begin{tikzpicture}[baseline=(current  bounding  box.center)]
				\matrix (m) [matrix of math nodes,row sep=4.5em,column sep=7em,minimum width=2em, text height=1.3ex, text depth=0.25ex]
				{
					T & & T \\
					& T \otimes_R P_{R/A}^m & \\
					R & &P_{T/A}^m. \\
				};
				\path[-stealth]
				(m-3-1) edge node [left] {$\varphi$} (m-1-1)
				(m-1-1) edge node [above] {$1$} (m-1-3)
				(m-1-1) edge node [above] {$\widehat{d}_{T/A}^m$} (m-2-2)
				(m-2-2) edge node [above] {$\phi^m$} (m-3-3)
				(m-3-1) edge node [above] {$1 \otimes_R d_{R/A}^m\quad\quad\;\;\;\,$} (m-2-2)
				(m-3-1) edge node [above] {$\phi^m \circ(1 \otimes_R d_{R/A}^m)$} (m-3-3)
				(m-3-3) edge node [right] {$\mu_{T/A}$} (m-1-3)
				(m-2-2) edge node [above] {$T \otimes_R\mu_{R/A}\quad\quad\;\;\;$} (m-1-3)
				;	
			\end{tikzpicture}	
		\end{center}
		Since the map $d_{T/A}^m \colon T \rightarrow P_{T/A}^m$ commutes with the outer square of this commutative diagram and the ring map $\varphi \colon R \rightarrow T$ is formally \'etale, we obtain the equality 
		$$
		d_{T/A}^m \;=\; \phi^m \circ \widehat{d}_{T/A}^m.
		$$ 
		The $T$-module $P_{T/A}^m$ is generated by the images $d_{T/A}^m(t)$ for all $t \in T$.
		Similarly, as $T \otimes_R P_{R/A}^m$ is generated as a $T$-module by the images $1 \otimes_R d_{R/A}^m(r)$ for all $r \in R$, it follows that $T \otimes_R P_{R/A}^m$ is also generated as a $T$-module by the image $\widehat{d}_{T/A}^m(t)$ for all $t \in T$.
		Therefore, the following equalities 
		\begin{equation*}
			\begin{array}{ccc}
				\widehat{d}_{T/A}^m &=\;\, \psi ^m \circ d_{T/A}^m &=\;\, (\psi^m \circ \phi^m) \circ \widehat{d}_{T/A}^m \\
				d_{T/A}^m &=\;\, \phi ^m \circ \widehat{d}_{T/A}^m &=\;\, (\phi^m \circ \psi^m) \circ d_{T/A}^m
			\end{array}
		\end{equation*}
		imply that $\psi^m \circ \phi^m = \text{id}_{T \otimes_R P_{R/A}^m}$ and $\phi^m \circ \psi^m = \text{id}_{P_{T/A}^m}$.
		This completes the proof of the proposition.
	\end{proof}

	\begin{theorem}
		\label{thm_diff_ops_etale}
		Let $\varphi\colon R \rightarrow T$ be a formally \'etale ring map of $A$-algebras.
		The following statements hold:
		\begin{enumerate}[\rm (i)]
			\item There is an induced $A$-algebra map $\varphi \colon D_{R/A} \rightarrow D_{T/A}$ such that, for any $\delta \in D_{R/A}$, the image $\delta' = \varphi(\delta)$ is the unique differential operator in $D_{T/A}$ that satisfies  
			$$
			\delta'\big(\varphi(r)\big) \;=\; \varphi\big(\delta(r)\big) \quad \text{ for all \quad $r \in R$}.
			$$
			
			\item If $P_{R/A}^m$ is a finitely presented $R$-module for all $m \ge 0$, then $\varphi \colon D_{R/A} \rightarrow D_{T/A}$ induces an isomorphism 
			$$
			T \otimes_R D_{R/A} \; \xrightarrow{\;\cong\;} D_{T/A}, \quad t \otimes_R \delta \,\mapsto\, t\varphi(\delta).
			$$		
		\end{enumerate}		
	\end{theorem}
	\begin{proof}
		(i)
		We have a natural map $D_{R/A}^m \rightarrow \Diff_{R/A}^m(R, T), \, \delta \mapsto \varphi \circ \delta$.
		By utilizing \autoref{prop_represen_diff_opp}, \autoref{prop_isom_princ_part_etale} and the Hom-tensor adjointness, we obtain the isomorphisms
		\begin{align*}
			\Diff_{R/A}^m(R, T) \,\cong\, \Hom_R(P_{R/A}^m, T) \,\cong\, \Hom_T(T \otimes_R P_{R/A}^m, T) \,\cong\, \Hom_T(P_{T/A}^m, T) \,\cong\, D_{T/A}^m.
		\end{align*}
		We define $\varphi^m$ as the composition $D_{R/A}^m \rightarrow \Diff_{R/A}^m(R, T) \xrightarrow{\cong} D_{T/A}^m$.
		Finally, the map $\varphi \colon D_{R/A} \rightarrow D_{T/A}$ is given by taking the direct limit $\varphi = \lim_{\rightarrow} \varphi^m$.
		The uniqueness assertion follows from the fact that $d_{T/A}^m$ and $\widehat{d}_{T/A}^m$ solve the same universal problem of representing differential operators in $D_{T/A}^m$ (see \autoref{prop_isom_princ_part_etale}).
		
		(ii) Since $\varphi \colon R \rightarrow T$ is flat (see, e.g., \cite[Th\'eor\`eme 19.7.1]{EGAIV_I}) and $P_{R/A}^m$ is a finitely presented $R$-module, we now obtain the isomorphisms 
		$$
		T \otimes_R D_{R/A}^m \,\cong\, T \otimes_R \Hom_R(P_{R/A}^m, R) \,\cong\,  \Hom_T(T \otimes_R P_{R/A}^m, T) \,\cong\, \Hom_T(P_{T/A}^m, T) \,\cong\, D_{T/A}^m.
		$$
		 As a consequence, the induced map $T \otimes_R D_{R/A}^m \rightarrow D_{T/A}^m$ is an isomorphism.
		 Finally, by taking a direct limit, it follows that the induced map $T \otimes_R D_{R/A} \rightarrow D_{T/A}$ is also an isomorphism.
	\end{proof}

	A distinguished type of formally \'etale ring map is localization. 
	The following remark describes the well-known behavior of differential operators under localization. 
	
	\begin{remark}[Localization of differential operators]
		\label{rem_localization_diff}
		Let $W \subset R$ be a multiplicatively closed subset.
		\begin{enumerate}[\rm (i)]
			\item If $P_{R/A}^m$ is finitely presented for all $m \ge 0$, then we obtain the natural isomorphisms:
			\begin{enumerate}[\rm (a)]
				\item $D_{W^{-1}R/(W \cap A)^{-1}A} \;\cong\; D_{W^{-1}R/A} \;\cong\; W^{-1}R \otimes_R D_{R/A}$.\smallskip
				\item $
				\Delta(W^{-1}R/(W \cap A)^{-1}A) \;\cong\; \Delta(W^{-1}R/A) \;\cong\; W^{-1}R \otimes_R \Delta(R/A).
				$
			\end{enumerate}
			\item Given $\delta \in D_{R/A}^m$, we extend it to an element $\delta' \in D_{W^{-1}R/(W \cap A)^{-1}A}$.
			We proceed by induction on $m$.
			If $m=0$, then $\delta \in \Hom_R(R,R)=R$ and $\delta'$ is defined by setting $\delta'(\frac{r}{w}) = \frac{\delta(r)}{w}$ for all $r \in R, w \in W$.	If $m > 0$, then we set
			$$ \qquad \delta'\Big(\frac{r}{w}\Big) \,=\, \frac{\delta(r) - [\delta,w]'(\frac{r}{w})}{w}
			\qquad \text{for all $r \in R$ and $w \in W$.}
			$$
			This is well-defined by the induction hypothesis and the fact that $[\delta,w] \in D_{R/A}^{m-1}$.
		\end{enumerate}
	\end{remark}
	\begin{proof}
		The result now follows from \autoref{thm_diff_ops_etale}.
		Alternatively, see \cite[Proposition 2.17]{BJNB}, \cite[\S 16]{EGAIV_IV}.
	\end{proof}

	We have the following finite generation results that are inspired by a quite pleasing result of Lyubeznik \cite{LYUBEZNIK_BS_POLY}.
	We first deal with a polynomial ring $P$ over the coefficient ring $A$ and then with \'etale $P$-algebras via \autoref{thm_diff_ops_etale}.

	\begin{proposition}
		\label{prop_finite_gen_loc}
		Let $A$ be Noetherian domain containing the field $\QQ$ of rational numbers and $P = A[x_1,\ldots,x_r]$ be a polynomial ring over $A$.
		Suppose that $\varphi \colon P \rightarrow R$ is an \'etale ring map of $A$-algebras.
		Then the following statements hold:
		\begin{enumerate}[\rm (i)]
			\item For any $f \in P$, there exists a nonzero element $a \in A$ such that $P_f \otimes_{A} A_a$ is a finitely generated left module over $D_{P/A} \otimes_{A} A_a$.
			\item For any $f \in R$, there exists a nonzero element $a \in A$ such that $R_f \otimes_{A} A_a$ is a finitely generated left module over $D_{R/A} \otimes_{A} A_a$.
		\end{enumerate}		
	\end{proposition}  
	\begin{proof}
		Let $D = D_{P/A}$, $\widehat{D} = D_{R/A}$, $K=\Quot(A)$, $P' = P \otimes_A K$ and $D' = D_{P'/K} \cong D_{P/A} \otimes_{A} K$.
		
		(i)	
		Let $f \in P$.	
		From the existence of the Bernstein--Sato polynomial (see \cite[Theorem 3.3, page 94]{COUTINHO}, \cite[Theorem 5.7, page 14]{BJORK}), we obtain a polynomial $b_f(s) \in K[s]$ (the Bernstein--Sato polynomial of $f$) and an operator $\delta(s) \in D'[s]$ that satisfy the following functional equation
		$$
		b_f(s)\, f^{s} = \delta(s)\, f^{s+1}.
		$$
		A known important result (see \cite[Proposition 2.11]{LEYKIN_BS_POLY}, \cite{KASHIWARA}, \cite{MALGRANGE}) says that $b_f(s)$ has rational roots; in particular,  $b_f(s) \in \QQ[s] \subset K[s]$.
		In \autoref{thm: BS rationality}, we present a proof of this result for any smooth algebra over a field of characteristic zero.
		Write $\delta(s) = \sum_{j=1}^l c_j \xx^{\alpha_j}\partial^{\beta_j}s^{\gamma_j} \in D'[s]$ with $c_j \in K$, $\alpha_j, \beta_j \in \NN^r$ and $\gamma_j \in \NN$.
		We can collect all the denominators in $A$ of the coefficients $c_j \in K$, and then localize at a suitable $0 \neq a \in A$ and assume that $\delta(s) \in D[s] \otimes_{A} A_a$.
		Let $\ell$ be a positive integer such that $b_f(k) \neq 0$ for all $k \le -\ell$.
		Therefore, 
		$$
		f^k \;=\; \frac{\delta(k)}{b_f(k)}f^{k+1} \;\in\; (D \otimes_A A_a) \cdot f^{k+1}
		$$ for all $k \le -\ell$, and by induction it follows that $P_f \otimes_{A} A_a$ is generated by $f^{-\ell}$ as a $(D \otimes_{A} A_a)$-module.
		
		(ii) 
		Let $f \in R$.
		Since $\Quot(P) \rightarrow R_f \otimes_P \Quot(P)$ is an \'etale ring map, it follows that $R_f \otimes_P \Quot(P)$ is isomorphic to the product of finitely many finite separable field extensions of $\Quot(P)$ (see, e.g., \cite[\href{https://stacks.math.columbia.edu/tag/00U3}{Tag 00U3}]{stacks-project}), and so, in particular, $R_f \otimes_P \Quot(P)$ is a finite dimensional vector space over $\Quot(P)$.
		Therefore, we may choose some nonzero element $g \in P$ such that $\varphi \colon P_g \rightarrow R_f$ is a module-finite \'etale ring map.
		
		From part (i), there exist a nonzero element $a \in A$ and a positive integer $\ell$ such that $P_g$ is generated by $g^{-\ell}$ as a $(D\otimes_A A_a)$-module.
		Consider an integral equation 
		$$
		\Big(\frac{1}{f}\Big)^m - c_{m-1}\Big(\frac{1}{f}\Big)^{m-1} - \cdots - c_{1}\Big(\frac{1}{f}\Big) - c_0 \,=\, 0  \; \in  R_f
		$$
		of $1/f$ over $P_g$, with $c_i \in P_g$.
		Hence, for all $k \ge 1$, we get the equality 
		$$
		\frac{1}{f^k} \,=\, \left( fc_{m-1} + \cdots + f^{m-1}c_1 + f^mc_0 \right)^k.
		$$
		This shows that $1/f^k$ can be written as a linear combination of elements of the form $f^ec$ with $e \ge 0$ and $c \in P_g$.
		Due to \autoref{thm_diff_ops_etale}, there is an isomorphism $
		R \otimes_P D \xrightarrow{\cong} \widehat{D},\,  r \otimes_R \delta \mapsto r\varphi(\delta).
		$
		By combining everything, we obtain a differential operator $\widehat{\delta}_k \in D_{R_a/A_a} \cong \widehat{D} \otimes_{A} A_a$ such that
		$$
		\frac{1}{f^k} \;=\; \widehat{\delta}_k\Big(\frac{1}{g^\ell
		}\Big).
		$$
		So, it follows that $R_f$ is generated by $\varphi(g^{-\ell})$ as a $(\widehat{D} \otimes_{A} A_a)$-module.
		Of course, as a consequence, there is some positive integer $\widehat{\ell}$ such that $R_f$ is generated by $f^{-\widehat{\ell}}$ as a $(\widehat{D} \otimes_A A_a)$-module.
	\end{proof}

	We derive some consequences for smooth algebras over the coefficient ring $A$, and for that purpose the following remark will be important.

	\begin{remark}
		\label{rem_cover_smooth_etale}
		Let $R$ be a smooth $A$-algebra.
		Then there exist finitely elements $g_1,\ldots,g_c \in R$ generating the unit ideal in $R$ and such that, for $1 \le i \le c$, there is a polynomial ring $P_i$ over $A$ such that $P_i \rightarrow R_{g_i}$ is an \'etale ring map.  
	\end{remark}
	\begin{proof}
		See \cite[\href{https://stacks.math.columbia.edu/tag/054L}{Tag 054L}]{stacks-project}.
	\end{proof}

	We are now ready for our main result regarding the behavior of differential operators in a relative smooth setting over a coefficient ring.

	\begin{theorem}
		\label{thm_gen_prop_smooth_A}
		Let $A$ be a Noetherian ring containing the field $\QQ$ of rational numbers and $R$ be a smooth $A$-algebra.
		Then the following statements hold: 
		\begin{enumerate}[\rm (i)]
			\item $D_{R/A} = \Delta(R/A)$. 
			In particular,
			$
			\gr(D_{R/A}) = \bigoplus_{m=0}^\infty D^{m}_{R/A}/D^{m-1}_{R/A}
			$ 
			is a Noetherian commutative ring and $D_{R/A}$ is a Noetherian ring.
			\item $D_{R/A}$ is strongly right Noetherian.
			\item Suppose that $A$ is a domain. 
			For any $f \in R$, there exists a nonzero element $a \in A$ such that $R_f \otimes_{A} A_a$ is a finitely generated left module over $D_{R/A} \otimes_{A} A_a$.
		\end{enumerate}
	\end{theorem}
	\begin{proof}	
		By \autoref{rem_cover_smooth_etale}, let $g_1,\ldots,g_c \in R$ be elements generating the unit ideal and such that $P_i \rightarrow R_{g_i}$ is an \'etale ring map and $P_i$ is a polynomial ring over $A$.
		
		(i) It suffices to show that $R_{g_i} \otimes_R D_{R/A} = R_{g_i} \otimes_R \Delta(R/A)$ for all $1 \le i \le c$.
		From \autoref{rem_localization_diff}, this is equivalent to check the equality $D_{R_{g_i}/A} = \Delta(R_{g_i}/A)$ for all $1 \le i \le c$.
		By applying \autoref{thm_diff_ops_etale} to the \'etale ring map $\varphi_i \colon P_i \rightarrow R_{g_i}$ we obtain an isomorphism 
		$$
		R_{g_i} \otimes_{P_i} D_{P_i/A}  \xrightarrow{\cong} D_{R_{g_i}/A}, \quad r \otimes_{P_i} \delta \mapsto r\varphi_i(\delta).
		$$
		Since $D_{P_i/A}$ is a relative Weyl algebra (see \autoref{rem_relative_Weyl_alg}), it is clear that $D_{P_i/A} = \Delta(P_i/A)$.
		As a consequence, we obtain the required equality $D_{R_{g_i}/A} = \Delta(R_{g_i}/A)$.
		The additional claims follow from \autoref{prop_derivation_ring}.
		
		(ii) Let $B$ be a Noetherian $A$-algebra. 
		By \autoref{lem_base_change_diff}, we have $D_{R/A} \otimes_{A} B \cong D_{(R\otimes_{A} B)/B}$, and so the already proved part (i) implies that $D_{R/A} \otimes_{A} B$ is Noetherian. 
		So it follows that $D_{R/A}$ is strongly right Noetherian.
		
		(iii) 
		By applying \autoref{prop_finite_gen_loc} to the \'etale ring map $P_i \rightarrow R_{g_i}$, we obtain a nonzero element $a_i \in A$ and a positive integer $\ell_i$ such that $\left(R_{g_i}\right)_f \otimes_{A} A_{a_i}= R_{a_ifg_i}$ is generated by $f^{-\ell_i}$ as a $(D_{R_{g_i}/A} \otimes_{A} A_{a_i})$-module.
		Suppose $k$ is an integer larger than all the $\ell_i$'s. 
		Therefore we may choose a differential operator $\delta_i/g_i^{e_i} \in D_{R_{a_ig_i}/A_{a_i}} = D_{R_{g_i}/A} \otimes_{A} A_{a_i}$, with $\delta_i \in D_{R_{a_i}/A_{a_i}} = D_{R/A} \otimes_A A_{a_i}$ and $e_i \ge 0$, such that 
		$$
		\frac{1}{f^k} \;=\; \frac{\delta_i}{g_i^{e_i}}\Big(\frac{1}{f^{\ell_i}}\Big) \;\in\; R_{a_ifg_i}.
		$$
		This induces the equality 
		$$
		\frac{g_i^{e_i}g_i^{s_i}}{f^k} \;=\; g_i^{s_i}\delta_i\Big(\frac{1}{f^{\ell_i}}\Big) \; \in \; R_{a_if}
		$$
		with some $s_i \ge 0$.
		Since the elements $g_1^{e_1+s_1}, \ldots, g_c^{e_c+s_c}$ also generate the unit ideal in $R$, we may find elements $\beta_1,\ldots,\beta_c \in R$ such that $\beta_1g_1^{e_1+s_1} + \cdots \beta_cg_c^{e_c+s_c} = 1$.
		By summing up, we obtain the equation 
		$$
		\frac{1}{f^k} \;=\; \frac{\beta_1g_1^{e_1+s_1} + \cdots +\beta_cg_c^{e_c+s_c}}{f^k} \;=\; \beta_1 g_1^{s_1}\delta_1\Big(\frac{1}{f^{\ell_1}}\Big) + \cdots + \beta_cg_c^{s_c} \delta_c\Big(\frac{1}{f^{\ell_c}}\Big) \; \in\; R_{a_1\cdots a_c f}. 
		$$
		Therefore, after taking $a = a_1\cdots a_c \in A$, it follows that $R_f \otimes_{A} A_a$ is a finitely generated left module over $D_{R/A} \otimes_{A} A_a$.
		This completes the proof of part (iii).
	\end{proof}

	\begin{corollary}
		\label{cor_gen_free_D_R/A}
		Adopt the same assumptions of \autoref{thm_gen_prop_smooth_A}.
		Then, for any finitely generated left $D_{R/A}$-module $M$, there is a nonzero element $a \in A$ such that $M \otimes_A A_a$ is a free $A_a$-module.
	\end{corollary}
	\begin{proof}
		The result follows from \autoref{thm_gen_prop_smooth_A} and \autoref{thm: noncom gen freeness}.
	\end{proof}

	\subsection{Classical case over a field of characteristic zero}
	
	In this subsection, we briefly cover the classical case of a smooth algebra over a field of characteristic zero. 
	Here the main goal is to show that  Bernstein--Sato polynomials have rational roots in any smooth algebra over a field of characteristic zero.
	Our approach is to utilize Kashiwara's result \cite{KASHIWARA} on local Bernstein--Sato polynomials and the techniques of Mebkhout and  Narv{\'a}ez-Macarro \cite{MNM} that allow us to globalize. 
	While the rationality result is usually stated for a polynomial ring,  the aforementioned approach can be used more generally. 
	We point out that this rationality result (in the polynomial ring case) was our main tool in the proof of \autoref{prop_finite_gen_loc}(i) and this more general result can be similarly used to give a different proof of \autoref{thm_gen_prop_smooth_A}(iii).
	
	We first recall some fundamental results about Bernstein--Sato polynomials. Our treatment follows the survey paper \cite{AMJNBSurvey}.

	\begin{definition}[\cite{NB}]
		Let $R$ be a Noetherian algebra over a field $\kk$ of characteristic $0$. 
		We say that $R$ is \emph{differentiably admissible} if $R$ is regular, 
		$\Der_{R/\kk}$ is a projective $R$-module of rank $\dim(R)$, 
		and for every maximal ideal $\mm$ of $R$ the following 
		three conditions are satisfied:
		\begin{enumerate}
			\item $\dim(R_\mm) = \dim(R)$,  
			\item $R/\mm$ is an algebraic extension of $\kk$, 
			\item the natural map $\Der_{R/\kk} \otimes_R R_\mm \to \Der_{R_\mm/\kk}$
			is an isomorphism.
		\end{enumerate}
	\end{definition}

	\begin{observation}
		\label{obs_equidim_smooth}
		If $R$ is an equidimensional smooth algebra over a field $\kk$ of characteristic $0$, then $R$ is differentiably admissible.
	\end{observation}
	\begin{proof}
		Since $R$ is finitely generated and equidimensional, the three conditions are satisfied. 
		We know that $\Der_{R/\kk}$ is projective because $R$ is smooth and its rank is constant due to \cite[\href{https://stacks.math.columbia.edu/tag/00TT}{Tag 00TT}]{stacks-project}.
	\end{proof}

	\begin{theorem}[{\cite[Theorem~3.26]{AMJNBSurvey}}]\label{thm: BS existence}
		Let $R$ be a differentiably admissible algebra over a field $\kk$ of characteristic $0$. 
		Then, for any $f \in R$, there exist a polynomial $b_f(s) \in \kk[s]$ {\rm(}the Bernstein--Sato polynomial{\rm)} and an operator $\delta(s) \in D_{R/\kk}[s]$ that satisfy the functional equation
		\[
		b_f(s)\, f^{s} = \delta(s)\, f^{s+1}.
		\]
	\end{theorem}

	By utilizing \autoref{thm_gen_prop_smooth_A}(i) and \cite[Proposition~2.10]{NB}, we obtain  $D_{R/\kk} = \Delta(R/\kk) = R \langle\Der_{R/\kk} \rangle$ when $R$ is a smooth algebra or a differentiably admissible algebra.
	This allows to easily define the following objects. 
	
	\begin{definition}
		Let $R$ be a smooth algebra or a differentiably admissible algebra over a field $\kk$ of characteristic $0$.  
		For an element $0 \neq f \in R$, we define $R_f[s] \fps$ as a free $R_f[s]$-module generated by the formal element $\fps$ and with a 
		$D_{R/\kk}[s]$-module structure determined as follows:
		for every derivation $\delta$ and $g \in R_f[s]$, we set
		\[
		\delta (g \fps) = \left (\delta (g) + \frac{sg \delta(f)}{f}  \right ) \fps.
		\]
		We define $D_{R/\kk}[s] \fps$ as the $D_{R/\kk}[s]$-submodule of 
		$R_f[s] \fps$ generated by $\fps$.
	\end{definition}
	
	These modules give a different interpretation of the Bernstein--Sato polynomial that was employed by Mebkhout and Narv{\' a}ez-Macarro \cite{MNM}.
	
	\begin{proposition}
		\label{prop: BS properties}
		Let $R$ be a smooth algebra over a field $\kk$ of characteristic $0$. 
		Then, for any $f \in R$, the following statements hold: 
		\begin{enumerate}[\rm (i)]
			\item The Bernstein--Sato polynomial $b_f(s)$ of $f$ exists.
			\item $b_f(s) = {\rm lcm}\big\lbrace b_f^{R_\mm}(s) \mid \mm \in \MaxSpec(R) \big\rbrace$, where $b_f^{R_\mm}(s)$ denotes the Bernstein--Sato polynomial over the localization at a maximal ideal $\mm \subset R$.
			\item For any field extension $\LL$ of $\kk$, $b_f(s)$ equals the Bernstein--Sato polynomial $b_f^{R\otimes_\kk \LL}(s)$ of $f \otimes 1 \in  R \otimes_\kk \LL$.
		\end{enumerate}
	\end{proposition}
	\begin{proof}
		First, notice
		that the Bernstein--Sato polynomial $b_f(s)$ of $f$, when it exists, 
		is the minimal polynomial of the action of $s$ on the $D_{R/\kk}[s]$-module 
		\[
		(D_{R/\kk}[s] \fps)/(D_{R/\kk}[s] f\fps).
		\] 
		
		(i) By \autoref{rem_cover_smooth_etale}, we can find elements $g_1,\ldots,g_c \in R$ generating the unit ideal and such that $P_i \rightarrow R_{g_i}$ is an \'etale ring map and $P_i$ is a polynomial ring over $\kk$.
		Since $R_{g_i}$ is regular, it is a finite product of regular domains $R_{g_i} = \prod_j R_{i,j}$. 
		Each domain $R_{i,j}$ is \'etale over $P_i$ (see \cite[\href{https://stacks.math.columbia.edu/tag/00U2}{Tag 00U2}]{stacks-project}), and thus should have the same dimension as $P_i$.
		This shows that $R_{g_i}$ is equidimensional.
		Then \autoref{obs_equidim_smooth} and \autoref{thm: BS existence} yield the Bernstein--Sato polynomial $b_{f}^{R_{g_i}}(s)$ and the vanishing
		$$
		 R_{g_i} \otimes_R  \left(b_{f}^{R_{g_i}}(s) \cdot
		(D_{R/\kk}[s] \fps)/(D_{R/\kk}[s] f\fps) \right) \,\cong\, b_{f}^{R_{g_i}}(s) \cdot (D_{R_{g_i}/\kk}[s] \fps)/(D_{R_{g_i}/\kk}[s] f\fps) \,=\, 0.
		$$
		As a consequence, we get  
		$$
		b_{f}^{R_{g_1}}(s) \cdots b_{f}^{R_{g_c}}(s) \,\cdot \, \big(D_{R/\kk}[s] \fps\big)/\big(D_{R/\kk}[s] f\fps\big) = 0,
		$$
		and this implies the existence of the Bernstein--Sato polynomial $b_f(s)$ of $f$.
		
		(ii)
		For any $\mm \in \MaxSpec(R)$, the isomorphism 
		\[
		R_\mm \otimes_R
		\big(D_{R/\kk}[s] \fps\big)/\big(D_{R/\kk}[s] f\fps\big) \,\cong\, \big(D_{R_\mm/\kk}[s] \fps\big)/\big(D_{R_\mm/\kk}[s] f\fps\big)
		\]
		shows that the local Bernstein--Sato polynomial divides the global one.  
		The assertion follows as we run through all the maximal ideals of $R$.
		
		(iii)
		Let $T = R \otimes_\kk \LL$.
		Let $\{e_i\}_{i \in I}$ be a basis of $\LL$ as a $\kk$-vector space.
		The inclusion $\kk \hookrightarrow \LL$ yields the isomorphisms 
		$$
		\frac{D_{T/\LL}[s] \fps}{D_{T/\LL}[s] (f \otimes 1)\fps} \;\cong \;
 		\LL \otimes_{\kk}  \frac{D_{R/\kk}[s] \fps}{D_{R/\kk}[s] f\fps}  \;\cong\;  \bigoplus_{i \in I} \frac{D_{R/\kk}[s] \fps}{D_{R/\kk}[s] f\fps} \, e_i.
		$$
		Hence $b_{f}^T(s) \subset \LL[s]$ divides $b_{f}(s) \subset \kk[s] \subset \LL[s]$.
		 We can write $b_{f}^T(s) = \sum b_i (s) e_i$ where $b_i(s) \in \kk[s]$ and only finitely many of them are not $0$. 		
		We obtain that $b_i(s) \cdot \big(D_{R/\kk}[s] \fps\big)/\big(D_{R/\kk}[s] f\fps\big) = 0$, and so it follows that $b_f(s)$ divides $b_i(s)$.
		This implies that $b_f(s)$ divides $b_{f}^T(s)$.
	\end{proof}

	Having proved the existence of Bernstein--Sato polynomials for a smooth algebra over a field of characteristic zero, we now present the following rationality result.
	
	\begin{theorem}[Kashiwara \cite{KASHIWARA}, Malgrange \cite{MALGRANGE}]\label{thm: BS rationality}
		Let $R$ be a smooth algebra over a field $\kk$ of characteristic $0$. 
		Then, for any $f \in R$, the Bernstein--Sato polynomial  $b_f(s)$ factors over $\QQ$.
	\end{theorem}
	\begin{proof}
		By applying \cite[\href{https://stacks.math.columbia.edu/tag/00TP}{Tag 00TP}]{stacks-project}, we can find a finitely generated field extension $\bFF$ of $\QQ$ and a smooth $\bFF$-algebra $R_0$ such that 
		$R \cong R_0 \otimes_{\bFF} \kk$. 
		Moreover, we may also assume that $f \in R_0$.
		We embed $\bFF$ into $\CC$ and consider $S \coloneqq R_0 \otimes_\bFF \CC$.
		Due to  the \autoref{prop: BS properties}(iii),  the Bernstein--Sato polynomial of $f$ in all three rings is the same one. 
		Thus, we reduced the problem to a smooth algebra $S$ over $\CC$. 
		The result of Kashiwara \cite{KASHIWARA} shows that the \emph{local} Bernstein--Sato polynomial (that is, the Bernstein--Sato polynomial for $f$ in the localization $S_\mm$ with $\mm \in \MaxSpec(S)$) can be factored completely in $\QQ[s]$. 
		The assertion now follows from \autoref{prop: BS properties}(ii).
	\end{proof}

  \section{Applications to local cohomology}
  \label{sect_local_cohom}
  	
  	In this section, we prove the following theorem regarding the generic freeness of local cohomology modules.

    	\begin{theorem}
    	\label{thm_local_cohom}
    	Let $A$ be Noetherian domain containing a field $\kk$ and $R$ be a smooth $A$-algebra. 
    	Suppose one of the following two conditions: 
    	\begin{enumerate}[\rm (a)]
    		\item $\kk$ is a field of characteristic zero, or
    		\item $\kk$ is a field of positive characteristic and the regular locus ${\rm Reg}(A) \subset \Spec(A)$ contains a nonempty open subset.
    	\end{enumerate}
    	Then, for any ideal $I \subset R$, there is a nonzero element $a \in A$ such that $\HH_I^i(R) \otimes_A A_a$ is a free $A_a$-module for all $i \ge 0$. 
    \end{theorem}
    \begin{proof}
		The result follows from \autoref{thm_gen_free_LC_zero} and  \autoref{thm_gen_free_LC_pos} below.
    \end{proof}

	As shown by the following example of Katzman \cite{KATZAMAN}, one cannot hope for a generic freeness result of local cohomology that does not involve additional assumptions.

		\begin{remark}
			\label{rem_Katzman}
		Let $\kk$ be a field and $S$ be the  $\kk$-algebra
		$$
		S = \frac{\kk[s,t,x,y,u,v]}{\left(sx^2v^2 - (t+s)xyuv+ty^2u^2\right)}.
		$$
		Consider $A=\kk[s,t]$ as the coefficient ring.
		In \cite[Example 4.6]{GEN_FREENESS_LOC_COHOM} it was proved that, {\it one cannot find an element $0\neq a \in A$ such that} 
		$$
		\HH_{(u,v)}^2(S) \otimes_{A} A_a
		$$ 
		is a free $A_a$-module.
	\end{remark}

  \subsection{Characteristic zero} 
  
  Here we prove the generic freeness of local cohomology modules in a characteristic zero setting. 
  The proof follows straightforwardly from the techniques and results we developed in \autoref{sect_diff_ops_smooth}.

   \begin{theorem}
	\label{thm_gen_free_LC_zero}
   Let $A$ be a Noetherian domain  containing the field of rational numbers $\QQ$ and  $R$ be a  smooth $A$-algebra. 
	Then, for any ideal $I \subset R$ and any $i \ge 0$, there exists a nonzero element $a \in A$ such that $\HH_I^i(R) \otimes_{A} A_a$ is a free $A_a$-module.
   \end{theorem}
   \begin{proof}
   	Let $D = D_{R/A}$.
   	Fix a set of generators $f_1,\ldots,f_n$ of $I$.
   	Notice that the \v{C}ech complex 
   	$$
   	{C}^\bullet \colon \quad 0 \rightarrow R \rightarrow \bigoplus_{i} R_{f_i} \rightarrow \bigoplus_{i < j} R_{f_if_j} \rightarrow \cdots \rightarrow R_{f_1\cdots f_n} \rightarrow 0
   	$$
   	is naturally a complex of left $D$-modules (see \autoref{rem_localization_diff}).
   	Due to \autoref{thm_gen_prop_smooth_A}(iii), after localizing at a suitable nonzero element in $A$, we may assume that $C^\bullet$ is a complex of finitely generated left $D$-modules.
	Therefore, \autoref{cor_gen_free_D_R/A} implies that $\HH_I^i(R) \cong \HH^i(C^\bullet)$ becomes $A$-free after localizing at another suitable nonzero element in $A$.
   \end{proof}

  \subsection{Positive characteristic}
  \label{subsect_pos_char}
  In positive characteristic, Lyubeznik introduced in \cite{Lyubeznik_FMOD} the theory of $F$-modules in order to show finiteness of associated primes of local cohomology. 
  We use the following setup throughout this subsection.
  
  \begin{setup}
  	\label{setup_pos_char}
  	Let $\kk$ be a field of characteristic $p > 0$ and $A$ be a Noetherian domain containing $\kk$.
  	Let $R$ be a smooth algebra over $A$.
	Define $R'$ to be an $R$-bimodule with usual structure on the left and the Frobenius endomorphism on the right: $r'r = r^pr'$ for all $r \in R$ and $r' \in R'$. 
  The Peskine--Szpiro functor is $F_R (M)\coloneqq R' \otimes_R M$. 
  \end{setup}
  
For example, $F_R(R) = R$ and $F_R(R/I) = R/{I^{[p]}}$. 
In general, if $M$ has a presentation $F_2 \xrightarrow{\phi} F_1 \to M \to 0$, then $F_2 \xrightarrow{\phi^{[p]}} F_1 \to F_R(M) \to 0$
where $\phi^{[p]}$ is obtained by taking all entries of the matrix $\phi$ to $p$th power. 

\begin{remark}
	\label{rem_regular_extension}
		Notice that, if we assume that $A$ is regular, then it follows that $R$ is also regular.
		Let $P \in \Spec(R)$ and $\pp = P \cap A \in \Spec(A)$.
		Since $A \rightarrow R$ is smooth, the local homomorphism $A_\pp \rightarrow R_P$ is flat and the fiber $R_P \otimes_{{A}_\pp}\kappa(\pp)$ is a regular ring.
		Hence \cite[Theorem 23.7]{MATSUMURA} implies that $R_P$ is a regular ring.
\end{remark}

Our main tool in a positive characteristic setting is the following result.
  
  \begin{theorem}[Lyubeznik, {\cite[Propositions 2.3, 2.10] {Lyubeznik_FMOD}}]
  	\label{thm_Lyub_roots}
  Suppose that $R$ is a regular ring, and let $I \subset R$ be an ideal. 
  Then $\lc_I^i (R)$ has a \emph{root}, i.e., a finitely generated $R$-module $M$ with an injective homomorphism $M \to F_R(M)$ such that 
  \[
	  \lc_I^i (R) \;=\; \lim_\to \left \{M \to F_R(M) \to F^2_R(M) \to \cdots \right \}.
  \]
  \end{theorem}

	By utilizing the above theorem and \autoref{cor_Grob_Frob_pow}, we obtain the following generic freeness result for local cohomology modules.

\begin{corollary}
	\label{cor_gen_free_LC_pos}
Suppose that $R$ is a regular ring, and let $I \subset R$ be an ideal. 
Then there exists a nonzero element $a \in A$ such that $\lc_I^i (R) \otimes_A A_a$ is a free $A_a$-module. 
\end{corollary}  
\begin{proof}
	From \autoref{thm_Lyub_roots}, we can find a root  $M$ of $\lc_I^i(R)$.
	Consider the short exact sequence $0 \rightarrow M \rightarrow F_R(M) \rightarrow Q \rightarrow 0$, with $Q$ some $R$-module.
	Write $Q = F/N$ with $F$ a free $R$-module of finite rank.
	Since $R$ is regular, the functor $F_R$ is exact by Kunz's theorem \cite{KUNZ_FROB}, and so we get a short exact sequence
	$$
	0 \rightarrow F_R^e(M) \rightarrow F_R^{e+1}(M) \rightarrow F_R^e(Q) \cong F/N^{[p^e]} \rightarrow 0
	$$
	for all $e \ge 0$.
	By utilizing  \autoref{thm_gen_grob} and \autoref{cor_Grob_Frob_pow}, we can find $0 \neq a \in A$ such that $M \otimes_A A_a$ and $F / N^{[p^e]} \otimes A_a$ (for all $e \ge 0$) are free $A_a$-modules.
	Therefore, for all $e \ge 0$, 
	$$
	0 \rightarrow F_R^e(M) \otimes_{A} A_a \rightarrow F_R^{e+1}(M) \otimes_{A} A_a \rightarrow F_R^e(Q)  \otimes_{A} A_a \rightarrow 0
	$$
	is a short exact sequence of free $A_a$-modules, and in particular, is $A_a$-split.
	As a consequence, we obtain that $\lc_I^i(R) \otimes_{A} A_a= \lim_\to \left(F_R^e(M) \otimes_{A} A_a\right)$ is a free $A_a$-module.
\end{proof}
  
  \begin{theorem}
  	\label{thm_gen_free_LC_pos}
	Assume \autoref{setup_pos_char} and suppose that ${\rm Reg}(A)$ contains a nonempty open subset of $\Spec(A)$.
	Let $I\subset R$ be an ideal.
	Then there exists a nonzero element $a \in A$ such that $\lc_I^i (R) \otimes_A A_a$ is a free $A_a$-module. 
  \end{theorem}
\begin{proof}
	The assumption yields a nonzero element $a' \in A$ such that $D(a') \subset {\rm Reg}(A)$,  hence $A_{a'}$ is regular. 
	By \autoref{rem_regular_extension}, $R \otimes_A A_{a'}$ is also regular.
	Then \autoref{cor_gen_free_LC_pos} gives a nonzero element $a'' \in A$ such that $\HH_I^i(R) \otimes_{A} A_{a'a''}$ is a free $A_{a'a''}$-module, and so result holds for $a = a'a'' \in A$.	
\end{proof}

	\begin{remark}
		\label{rem_reg_locus}
		If we assume that $A$ is a finitely generated algebra over $\kk$, from \cite[\href{https://stacks.math.columbia.edu/tag/07PJ}{Tag 07PJ}]{stacks-project} and \cite[\href{https://stacks.math.columbia.edu/tag/07P7}{Tag 07P7}]{stacks-project}, we obtain that
		$$
		\text{Reg}(A) = \big\{ \pp \in \Spec(A) \mid A_\pp \text{ is a regular local ring}\big\}
		$$
		is an open subset of $\Spec(A)$.
		Moreover, as $A$ is a domain, $\text{Reg}(A)$ is non-empty.
	\end{remark}
  
	\subsection{Finiteness of associated primes}  
  
	We apply our result to derive a finiteness property of associated primes 
	for smooth algebras over a regular ring of dimension one. As far as we know this result is new, its analogues hold in mixed characteristic by \cite[Theorem~4.3]{BBLSZ} and in positive characteristic by \cite{Lyubeznik_FMOD}.

  \begin{theorem}\label{thm_finite_ass}
  Let $A$ be a Dedekind domain of characteristic $0$ and $R$ be a smooth $A$-algebra. Then for any ideal $I$ of $R$ and any nonnegative integer $i \ge 0$  the module $\lc_I^i (R)$ has finitely many associated primes. 
  \end{theorem}
  \begin{proof}
  By \autoref{thm_gen_free_LC_zero} there is $0 \neq a \in A$ such that $\lc_I^k (R)_a$ is a free $A_a$-module. 
  It follows that $\Ass_A (\lc_I^i (R))$ is finite. 
  Therefore it suffices to show that there are only finitely many $R$-associated primes contracting to a given $A$-associated prime. We proceed as in \cite[Theorem~4.3]{BBLSZ}. 
  
First, there are only finitely many associated primes that contract to $(0) \subset A$ since 
$\lc_I^i (R) \otimes_A \Quot(A) \cong  \lc_I^i (R \otimes_A \Quot(A))$ and 
$R \otimes_A \Quot(A)$ is a regular finitely generated algebra over a field of characteristic $0$ (\cite[Remark~3.7 (i)]{Lyubeznik_DMOD}). Second, for any maximal ideal 
$\pp$ of $A$ the localization $A_\pp$ is a DVR, let $\pi$ be a uniformizer of $A_\pp$. 
The exact sequence 
\[0 \to A_\pp \xrightarrow{\pi} A_\pp \to \kappa(\pp) \to 0\] 
induces the exact sequence
\[\lc_I^{i-1}(R \otimes_A \kappa(\pp)) \xrightarrow{d} \lc_I^{i}(R_\pp) \xrightarrow{\pi} \lc_I^i (R_\pp).
\]
So any associated prime of   $\lc_I^i (R_\pp)$ that contracts to $\pp$ 
must be an associated prime of $\IM(d)$.
By smoothness, $d$ is a map of modules over $D_{R/A} \otimes_A \kappa(\pp) \cong D_{(R \otimes_A k(\pp))/\kappa(\pp)}$. Thus,  because $\lc_I^i (R \otimes_A k(\pp))$ has finite length by \cite{Lyubeznik_DMOD}, 
$\IM(d)$ has finite length as a $D_{(R \otimes_A k(\pp))/\kappa(\pp)}$-module, hence it can only have finitely many associated primes.
  \end{proof}
  
  \section{Examples: specializations of determinantal ideals}
  \label{sect_determinant}
  
By utilizing the methods developed in this paper, we show how to obtain specializations of determinantal ideals that keep many of the original invariants of determinantal ideals.
We chose this example because a description of a Gr\"obner basis is well-known in this case. 
One could use these techniques with any family of ideals for which we have an explicit computation of Gr\"obner bases (e.g., for Schubert determinantal ideals \cite[\S 16.4]{MILLER_STURMFELS}).
We use the following setup throughout.

\begin{setup}
	\label{setup_det}
	Let $A$ be a Noetherian domain containing a field $\kk$ and $R = A\left[x_{i,j} \mid 1 \le i \le m \text{ and } 1 \le j \le n\right]$ be a polynomial ring with $m \le n$.
	For each $1 \le i \le m$ and $1 \le j \le n$, let $a_{i,j} \in A$ be a nonzero element.
	Consider the matrix 
	$$
	M \;=\; \left[
	\begin{array}{cccc}
		a_{1,1}x_{1,1} & a_{1,2}x_{1,2} & \cdots &  a_{1,n}x_{1,n}\\
		a_{2,1}x_{2,1} & a_{2,2}x_{2,2} & \cdots &  a_{2,n}x_{2,n}\\
		\vdots & \vdots & \ddots &  \vdots\\
		a_{m,1}x_{m,1} & a_{m,2}x_{m,2} & \cdots &  a_{m,n}x_{m,n}
	\end{array}  
	\right] \in R^{m \times n}
	$$
	and let $I_t = I_t(M) \subset R$ be the corresponding ideal of $t$-minors, where $1 \le t \le m$.
	Let $[m] = \{1,\ldots,m\}$ and $[n] = \{1,\ldots,n\}$.
	Let $\mathcal{H}$ be the set 
	$$
		\mathcal{H} \;=\; \big\lbrace (i,j) \in [m] \times [n] \mid i+j \le t-1 \big\rbrace  \;\;\bigcup\;\;
		 \big\lbrace (i,j) \in [m] \times [n] \mid   i + j \ge n+m-t+2  \big\rbrace 
	$$
	of positions not in the main antidiagonal of any $t$-minor of the matrix $M$.
\end{setup}

We have the following specific generic freeness result. 

\begin{theorem}
	\label{thm_gen_free_detertminants}
	Assume \autoref{setup_det}.
	Choose $0 \neq a \in \bigcap_{(i,j) \notin \mathcal{H}} \left(a_{i,j}\right) \subset A$.
	Then the following statements hold: 
	\begin{enumerate}[\rm (i)]
		\item $R/I_t \otimes_{A} A_a$ is a free $A_a$-module.
		\item $\HL^i(R/I_t) \otimes_{A} A_a$ is a free $A_a$-module for all $i \ge 0$.
	\end{enumerate}
\end{theorem}
 \begin{proof}
	Consider the polynomial ring $T = \kk\left[y_{i,j} \mid 1 \le i \le m \text{ and } 1 \le j \le n\right]$ and the generic $n\times m$ matrix $Y = \left(y_{i,j}\right)$.
	We have an injective $\kk$-algebra map $\varphi : T \hookrightarrow R,\, y_{i,j} \mapsto a_{i,j}x_{i,j}$ such that $I_t = \varphi\left(I_t(Y)\right)R$.
	
	Let $>$ be an antidiagonal monomial order on $Y$ (see \cite[Chapter 4]{BCRM}, \cite[\S 16.4]{MILLER_STURMFELS}).
	We extend the same monomial order to $R$ (in the sense of \autoref{setup_Grobner_deform}) by saying that $x_{i,j} < x_{i',j'}$ if and only if $y_{i,j} < y_{i',j'}$.
	A Gr\"obner basis of $I_t(Y)$ is given by the $t$-minors of $Y$ (see \cite[Chapter 4]{BCRM}, \cite[\S 16.4]{MILLER_STURMFELS}). 
	We substitute $A$ by $A_a$. 
	Let $v_1,\ldots,v_b$ be the $t$-minors of $Y$ and $w_1 = \varphi(v_1),\ldots,w_b = \varphi(v_b)$ be the $t$-minors of $M$.
	The elements $\iniTerm(v_i)$ and $\iniTerm(w_i)$ are generated by the product of the elements in the main antidiagonal of a $t$-minor of the matrices $Y$ and $M$, respectively.
	The syzygies $\Syz(\iniTerm(v_1),\ldots,\iniTerm(v_b))$ are generated by the divided Koszul relations (see \cite[Lemma 15.1]{EISEN_COMM}).
	Since we are assuming that $a_{i,j}$ is invertible when $(i,j) \notin \mathcal{H}$, it follows that the syzygies  $\Syz\left(\iniTerm(w_1),\ldots,\iniTerm(w_b)\right)$ are generated by image of $\Syz(\iniTerm(v_1),\ldots,\iniTerm(v_b))$ under the induced map $\varphi \colon T^b \rightarrow R^b$.
	Therefore, as $v_1,\ldots,v_b$ is a Gr\"obner basis for $I_t(Y)$, by utilizing \autoref{thm_Buchberger_crit} we obtain that $w_1,\ldots,w_b$ is also a Gr\"obner basis for $I_t$.
	
	Part (i) now follows from \autoref{thm_gen_grob}.
	Since each $\iniTerm(w_i)$ is square-free, \autoref{thm_sqr_free_deg} yields the result of part (ii).
 \end{proof}

The next corollary gives a specific locus where a natural specialization behaves just as a generic matrix. 
 
 \begin{corollary}
		Let $S = \kk\left[x_{i,j} \mid 1 \le i \le m \text{ and } 1 \le j \le n\right]$ and $X = (x_{i,j})$ be the $m \times n$ generic matrix.
		Consider the matrix 
		$$
		N \;=\; \left[
		\begin{array}{cccc}
			\beta_{1,1}x_{1,1} & \beta_{1,2}x_{1,2} & \cdots &  \beta_{1,n}x_{1,n}\\
			\beta_{2,1}x_{2,1} & \beta_{2,2}x_{2,2} & \cdots &  \beta_{2,n}x_{2,n}\\
			\vdots & \vdots & \ddots &  \vdots\\
			\beta_{m,1}x_{m,1} & \beta_{m,2}x_{m,2} & \cdots &  \beta_{m,n}x_{m,n}
		\end{array}  
		\right] \in S^{m \times n}
		$$
		where $\beta_{i,j} \in \kk$.
		If $\beta_{i,j} \neq 0$ for all $(i,j) \notin \mathcal{H}$, then the following statements hold:
		\begin{enumerate}[\rm (i)]
			\item $\dim_\kk\left(\left[S/I_t(N)\right]_\mu\right) = \dim_\kk\left(\left[S/I_t(X)\right]_\mu\right)$ for all $\mu \in \ZZ$.
			\item $\dim_\kk\left(\left[\HL^i(S/I_t(N))\right]_\mu\right) = \dim_\kk\left(\left[\HL^i(S/I_t(X))\right]_\mu\right)$ for all $i \ge 0, \mu \in \ZZ$.
		\end{enumerate}
 \end{corollary}
\begin{proof}
	Here we specify \autoref{setup_det} by setting that $A = \kk[a_{i,j}]$ is a polynomial ring over $\kk$.
	Then $N$ becomes the specialization of $M$, under the natural specialization map $\pi \colon R \rightarrow R \otimes_{A} A/\pp \cong S$, where $\pp = \left(a_{i,j} - \beta_{i,j}\right) \in \Spec(A)$ is a rational maximal ideal in $A$. 
	The result of the corollary follows from \autoref{thm_gen_free_detertminants} and the base change property of fiber-full modules (see \autoref{thm_fib_full_mod}).
\end{proof}

\section*{Acknowledgments}  

We thank Linquan Ma for suggesting \autoref{thm_finite_ass}.
We thank Gennady Lyubeznik for pointing out \cite[Example 3.3]{BBLSZ}.
We thank Wenliang Zhang for discussions. 
We thank the reviewers for carefully reading our paper and for their comments and corrections.

Y.C.R. appreciates the hospitality offered by the Singularity Theory and Algebraic Geometry group at BCAM, where this work was initiated.
 
I.S. was supported by the Spanish Ministry of Science, Innovation and Universities
through the project PID2021-125052NA-I00 and the 
grant RYC2020-028976-I funded by MCIN/AEI/10.13039/501100011033 and by FSE ``invest in your future''.
 
 \bibliography{references}

\end{document}